\documentclass[english]{iopart}
\usepackage[T1]{fontenc}
\usepackage[latin9]{inputenc}
\usepackage{color}
\usepackage{babel}
\usepackage{float}
\usepackage{units}
\usepackage{amsbsy}
\usepackage{amstext}
\usepackage{amsthm}
\usepackage{graphicx}
\usepackage[unicode=true,
 bookmarks=true,bookmarksnumbered=true,bookmarksopen=false,
 breaklinks=true,pdfborder={0 0 0},backref=false,colorlinks=true]
 {hyperref}

\makeatletter

\floatstyle{ruled}
\newfloat{algorithm}{tbp}{loa}
\providecommand{\algorithmname}{Algorithm}
\floatname{algorithm}{\protect\algorithmname}

\usepackage{iopams}
\usepackage{setstack}
\theoremstyle{plain}
\newtheorem{thm}{\protect\theoremname}
  \theoremstyle{plain}
  \newtheorem{prop}[thm]{\protect\propositionname}
  \theoremstyle{plain}
  \newtheorem{lem}[thm]{\protect\lemmaname}

\usepackage{iopams}

\newcommand{\eqref}[1]{(\ref{#1})}

\usepackage{algorithmic} 
\usepackage{algorithm} 

\@ifundefined{showcaptionsetup}{}{%
 \PassOptionsToPackage{caption=false}{subfig}}
\usepackage{subfig}
\makeatother

  \providecommand{\lemmaname}{Lemma}
  \providecommand{\propositionname}{Proposition}
\providecommand{\theoremname}{Theorem}

\begin{document}

\review{Total Variation Superiorized Conjugate Gradient Method for Image
Reconstruction}

\author{Marcelo V. W. Zibetti$^{\text{1}}$, Chuan Lin$^{\text{2}}$ and
Gabor T. Herman$^{\text{3}}$ }

\address{$^{\text{1}}$Graduate Program in Electrical Engineering and Computer
Science, Federal University of Technology - Paraná, Curitiba, 80230-901,
Brazil}

\address{$^{\text{2}}$School of Electrical Engineering, Southwest Jiaotong
University, Chengdu, Sichuan 610031, China}

\address{$^{\text{3}}$Department of Computer Science, The Graduate Center,
City University of New York, NY 10016, USA}

\ead{marcelozibetti@utfpr.edu.br, lin\_langai@163.com, gabortherman@yahoo.com}
\begin{abstract}
The conjugate gradient (CG) method is commonly used for the relatively-rapid
solution of least squares problems. In image reconstruction, the problem
can be ill-posed and also contaminated by noise; due to this, approaches
such as regularization should be utilized. Total variation (TV) is
a useful regularization penalty, frequently utilized in image reconstruction
for generating images with sharp edges. When a non-quadratic norm
is selected for regularization, as is the case for TV, then it is
no longer possible to use CG. Non-linear CG is an alternative, but
it does not share the efficiency that CG shows with least squares
and methods such as fast iterative shrinkage-thresholding algorithms
(FISTA) are preferred for problems with TV norm. A different approach
to including prior information is superiorization. In this paper it
is shown that the conjugate gradient method can be superiorized. Five
different CG variants are proposed, including preconditioned CG. The
CG methods superiorized by the total variation norm are presented
and their performance in image reconstruction is demonstrated. It
is illustrated that some of the proposed variants of the superiorized
CG method can produce reconstructions of superior quality to those
produced by FISTA and in less computational time, due to the speed
of the original CG for least squares problems. In the Appendix we
examine the behavior of one of the superiorized CG methods (we call
it S-CG); one of its input parameters is a positive number $\varepsilon$.
It is proved that, for any given $\varepsilon$ that is greater than
the half-squared-residual for the least squares solution, S-CG terminates
in a finite number of steps with an output for which the half-squared-residual
is less than or equal to $\varepsilon$. Importantly, it is also the
case that the output will have a lower value of TV than what would
be provided by unsuperiorized CG for the same value $\varepsilon$
of the half-squared residual.
\end{abstract}



\submitto{Inverse Problems.}

\maketitle

\section{\label{sec:Introduction}Introduction}

Solving least squares problems with linear models has a long history
in image reconstruction. The least squares (LS) problem for image
reconstruction is stated as: \textbf{find}
\begin{equation}
\mathbf{x}_{LS}=\arg\min_{\mathbf{x}}\frac{1}{2}\left\Vert \mathbf{y-Ax}\right\Vert _{2}^{2},\label{eq:LS1}
\end{equation}
where $\mathbf{y}$ is an $L$-dimensional vector containing the captured
data, usually the sinogram in computed tomography (CT), $\mathbf{x}$
is a $J$-dimensional vector containing the image to be reconstructed,
and $\mathbf{A}$ is an $L\times J$ system matrix, used for discrete
approximation of the line integrals or the Radon transform in CT.

The LS problem is convex and its solution satisfies the linear system
\begin{equation}
\mathbf{A}^{T}\mathbf{A}\mathbf{x}_{LS}=\mathbf{A}^{T}\mathbf{y}.\label{eq:LS2}
\end{equation}
This is a widely understood approach to the LS problem, with many
tools available from linear algebra to solve it or analyze it \cite{Bjorck1996,Barrett-2004}.
LS is also strongly connected to Maximum Likelihood (ML) estimation
for Gaussian noise models \cite{Kim2008}, which is a well-understood
topic from statistics.

Due to the large dimensionality of the problems in image reconstruction,
direct solving of (\ref{eq:LS2}) is not practical and iterative methods
are usually preferred. The conjugate gradient (CG) method \cite{Vogel-2002}
is a good choice since it converges very rapidly (just a few iterations
often provide acceptable results) and frequently it can use fast computational
operators instead of numerical matrices. When the system is ill-conditioned,
it is still possible to use preconditioning, getting even faster convergence. 

In image reconstruction problems (where $\mathbf{A}^{T}\mathbf{A}$
is singular or ill-conditioned and data are noisy) regularization
should be utilized. This can be done by adding a penalty to the right
hand side of (\ref{eq:LS1}) (commonly used penalties are total variation
(TV) norm and sparsity promoting norms, such as the $\ell_{1}$-norm
of a transform of $\mathbf{x}$; see, for example, \cite{Bovik-2000,Romberg2008}).
This results in changing (\ref{eq:LS1}) to
\begin{equation}
\mathbf{x}_{RLS}=\arg\min_{\mathbf{x}}\left(\frac{1}{2}\left\Vert \mathbf{y-Ax}\right\Vert _{2}^{2}+\lambda R(\mathbf{x})\right),\label{eq:RLS1}
\end{equation}
whose solution satisfies
\begin{equation}
\mathbf{A}^{T}\mathbf{A}\mathbf{x}_{RLS}+\lambda\nabla R(\mathbf{x}_{RLS})=\mathbf{A}^{T}\mathbf{y},\label{eq:RLS2}
\end{equation}
which is a nonlinear system of equations. 

One method for solving (\ref{eq:RLS2}) is nonlinear CG (NLCG), which
is a direct generalization of CG for solving nonlinear systems \cite{Dai2000,Hager2006}.
However, if $R(\mathbf{x})$ is an $\ell_{1}$-norm or the TV norm,
then NLCG may face some issues, such as non-differentiability and
the lack of a specific fast line search. 

A way of dealing with non-differentiability is to use subgradients
\cite[Chapter 8]{Rockafellar1998}, but this may put the optimization
process into trouble, since the negative of a subgradient is not necessarily
a descent direction. This difficulty can be overcome by using a strictly
convex approximation that replaces the non-differentiable norm with
a differentiable one, such as the Hyperbola function or the Huber
norm \cite{Zhao-2012}. For example, if $R(\mathbf{x})=||\mathbf{\mathbf{R\mathbf{x}||_{1}}}$,
where $\mathbf{R}$ is an $n\times J$ matrix, then $\nabla R(\mathbf{x})\approx\mathbf{\mathbf{R}}^{T}\mathbf{W\mathbf{\mathbf{R\mathbf{x}}}}$
for
\begin{equation}
\mathbf{W}=diag\left(\frac{1}{\sqrt{w_{1}^{2}+\kappa^{2}}},\cdots,\frac{1}{\sqrt{w_{n}^{2}+\kappa^{2}}}\right),\label{eq:rounding1}
\end{equation}
where $\mathbf{w}=\left(w_{1,}\ldots,w_{n}\right)^{T}=\mathbf{R}\mathbf{x}$
and $\kappa$ a small positive number. For the lack of a specific
line search in NLCG, we note that some alternatives exist when $\ell_{1}$-norm
is used \cite{Zibetti2016a}, but for other cases general line searches,
such as back tracking (which may be costly), has to be used \cite{Vogel-2002}.
Such considerations indicate that NLCG is not as good a choice for
solving (\ref{eq:RLS1}) as CG is for solving the least squares problem
(\ref{eq:LS1}).

For such reasons, other approaches are preferred. One of these is
forward-backward splitting \cite{Combettes2011}, which is utilized
by the iterative shrinkage-thresholding algorithm (ISTA) \cite{Zibulevsky2010},
TWIST \cite{Bioucas-Dias2007} and fast ISTA (FISTA) \cite{Beck2009}
that are based on proximal operators and Nesterov acceleration. In
these approaches, the gradient of only the quadratic part of (\ref{eq:RLS1})
is utilized and is combined with the proximal operator for the non-differentiable
part $R(\mathbf{x})$. In ISTA, we can write the iterations as
\begin{equation}
\mathbf{x}_{k+1}=\text{prox}_{\nicefrac{\lambda}{c^{2}},R(\mathbf{x})}\left(\mathbf{x}_{k}+\frac{1}{c^{2}}\mathbf{A}^{T}(\mathbf{\mathbf{y}-A}\mathbf{x}_{k})\right),\label{eq:ISTA}
\end{equation}
where the proximal operator is defined as
\begin{equation}
\text{prox}_{\alpha,R(\mathbf{x})}(\mathbf{b})=\arg\min_{\mathbf{x}}\left(\frac{1}{2}\left\Vert \mathbf{b-x}\right\Vert _{2}^{2}+\alpha R(\mathbf{x})\right).\label{eq:prox}
\end{equation}

The proximal operator has been demonstrated to deal well with the
non-differentiability of the $\ell_{1}$-norm and TV. Proximal operators
look like denoising \cite{Parikh2013}, but they are connected with
the subdifferential and are known to satisfy (\ref{eq:RLS2}) at convergence
if $\nabla R(\mathbf{x})$ is a subgradient of $R(\mathbf{x})$. Proximal
operators for some $R(\mathbf{x})$ are quite simple to implement
in practice; for example, the proximal operator of the $\ell_{1}$-norm
can be found by soft-thresholding or shrinkage-thresholding. Other
cases, including total variation, are not so easy. For TV, some recent
approaches using a dual formulation \cite{Chambolle2004,Chambolle2011}
allow us to use iterative algorithms, such as the gradient projection
(GP) and the fast gradient projection (FGP) \cite{Beck2009a} algorithms
as an iterative solution for the proximal operator. 

Even though regularization is very commonly used, it is not the only
way of including prior information into the solution; a more-recently
introduced approach is superiorization \cite{Censor2010,Herman2012}.
With a similar purpose to Tikhonov regularization \cite{Vogel-2002,Hansen-1998},
superiorization makes use of extra information about the desired image
to improve the reconstructed result. However, instead of adding a
side penalty to the original optimization problem, superiorization
preserves the original iterative algorithm and includes perturbation
steps between iterations. These steps shift the solutions obtained
along the iterations to a better path of solutions; ones that are
superior according to a secondary criterion, while preserving performance
according to the original optimization criterion. Published proofs
of such behavior assume that the iterative algorithm that is being
superiorized has the property of being ``perturbation resilient'';
see, for example, \cite{Censor2010,Herman2012}.

Regularization changes an originally ill-posed problem into a different
well-posed problem, that may require a new optimization algorithm
for solving it. (This was the case when regularizing least squares
with sparsity priors.) On the other hand, superiorization is an algorithm-preserving
approach, which perturbs the original algorithm towards a better path
of intermediate solutions. Considering a general algorithmic step
$\mathbf{x}_{k+1}=\mathbf{P}(\mathbf{x}_{k})$, superiorization changes
it to

\begin{equation}
\mathbf{x}_{k+1}=\mathbf{P}(\mathbf{x}_{k}+\gamma_{k}\mathbf{v}_{k}),\label{eq:superiorization}
\end{equation}
where $\mathbf{v}_{k}$ is a nonascending direction for $R(\mathbf{x})$
(which might be obtained using gradient, proximal operator or something
else) and $\gamma_{k}$ is a step size that ensures that the next
point is superiorized according to the cost provided by the secondary
criterion, i.e., $R(\mathbf{\mathbf{x_{\mathit{k}}}+\gamma_{\mathit{k}}\mathbf{v}_{\mathit{k}}})\leq R(\mathbf{x_{\mathit{k}}})$.
Therefore we do not need an essentially new algorithm, but rather
just the inclusion of superiorization steps, which shift the solution
of the original algorithm to a better solution at each iteration.
When appropriate conditions are satisfied, both regularization and
superiorization provide good reconstructed images. 

Superiorization has been proven successful in many applications \cite{Censor2010,Herman2012,Censor2014,Garduno2014,Schrapp2014,Helou2016}.
In this paper, superiorization is applied to the conjugate gradient
algorithm to solve least squares problems, using the proximity function\footnote{As defined in the superiorization literature, the proximity function
$f$ is a measure of how badly the constraints are violated; it is
not to be confused with the proximal operator as defined in (\ref{eq:prox}). } $f$ that is defined by the half-squared-residual $f(\mathbf{x})=\frac{1}{2}\left\Vert \mathbf{y-Ax}\right\Vert _{2}^{2}$
as the primary criterion and total variation as the secondary criterion
(see, for example, \cite{Herman2012}). Our aim was to keep the good
speed of CG, in spite of it being perturbed for superiorization.

It is demonstrated in this paper that there are a number of ways to
achieve our aim; we propose five different versions and some of which
are very effective. Experimental results illustrate superiorized CG
producing reconstructed images of similar quality to those produced
by FISTA, but faster.

The paper is organized as follows. Section \ref{sec:CT_Problem} presents
the CG method and an introduction to superiorization. In Section \ref{sec:Superiorization}
the proposed superiorized CG methods are presented. In Section \ref{sec:Experiments}
the results of some experiments comparing various methods are shown.
The conclusions and a discussion are in Section \ref{sec:Conclusions}.
A convergence proof of one of the proposed superiorized CG methods
is presented in the Appendix.

\section{\label{sec:CT_Problem}Review of Methods}

\begin{algorithm}[!tbp]
\caption{Standard Linear Conjugate Gradient $\mathrm{CG}(\mathbf{A},\mathbf{y},\mathbf{x}_{0},\varepsilon,\mathbf{x})$}

\label{CG}

\begin{algorithmic}[1]

\STATE{\textbf{set} $k=0$}

\STATE{\textbf{set }$\mathbf{g}_{0}=\mathbf{A}^{T}(\mathbf{A}\mathbf{x}_{0}-\mathbf{y})$}

\STATE{\textbf{set }$\mathbf{p}_{0}=-\mathbf{g}_{0}$}

\STATE{\textbf{set} $\delta_{0}=\left\Vert \mathbf{g}_{0}\right\Vert _{2}^{2}$}

\STATE{\textbf{while }$f(\mathbf{x}_{k})>\varepsilon$ }\label{sc}

\STATE{~~~~\textbf{call} CG-STEP$(\mathbf{A},\mathbf{x}_{k},\mathbf{p}_{k},\mathbf{g}_{k},\delta_{k},\mathbf{x}_{k+1},\mathbf{p}_{k+1},\mathbf{g}_{k+1},\delta_{k+1})$}\label{CG-call-CG-STEP}

\STATE{~~~~\textbf{set $k=k+1$}} 

\STATE{\textbf{set} $\mathbf{x=x}_{k}$ } 

\end{algorithmic}
\end{algorithm}

\medskip{}

\begin{algorithm}[!tbp]
\caption{CG-STEP$(\mathbf{A},\mathbf{x},\mathbf{p},\mathbf{g},\delta,\mathbf{x}',\mathbf{p}',\mathbf{g}',\delta')$}

\label{CG-STEP}

\begin{algorithmic}[1]

\STATE{\textbf{set} $\mathbf{h}=\mathbf{A}^{T}\mathbf{A}\mathbf{p}$}

\STATE{\textbf{set} $\alpha=\delta/\mathbf{p}^{T}\mathbf{h}$}\label{set alpha_k}

\STATE{\textbf{set }$\mathbf{x}'=\mathbf{x}+\alpha\mathbf{p}$}

\STATE{\textbf{set }$\mathbf{g}'=\mathbf{g}+\alpha\mathbf{h}$}

\STATE{\textbf{set }$\delta'=\left\Vert \mathbf{g}'\right\Vert _{2}^{2}$}

\STATE{\textbf{set }$\beta=\delta'/\delta$}\label{set beta_k}

\STATE{\textbf{set }$\mathbf{p}'=\mathbf{-g}'+\beta\mathbf{p}$}

\end{algorithmic}
\end{algorithm}

\medskip{}
The conjugate gradient algorithm (CG) is an effective way of solving
(\ref{eq:LS2}); its details are provided in Algorithms \ref{CG}
and \ref{CG-STEP}. In the parameter list of the algorithm $\mathrm{CG}(\mathbf{A},\mathbf{y},\mathbf{x}_{0},\varepsilon,\mathbf{x})$,
the input parameters are the matrix $\mathbf{A}$ and the vector $\mathbf{y}$
(as specified immediately after (\ref{eq:LS1})), the initial guess
$\mathbf{x}_{0}$ at the solution (a $J$-dimensional vector), the
error $\varepsilon$ that we are willing to tolerate according to
the primary criterion and the output parameter is $\mathbf{x}$, which
is an approximation to the desired $\mathbf{x}_{LS}$ of (\ref{eq:LS1}).

Using $k$ as the iteration counter, CG terminates as soon as $f(\mathbf{x}_{k})\leq\varepsilon$
(recall that $f(\mathbf{x})=\frac{1}{2}\left\Vert \mathbf{y-Ax}\right\Vert _{2}^{2}$
is the primary criterion), where $\varepsilon$ is a user-specified
input parameter. Given the user-specified initial guess $\mathbf{x}_{0}$
at the solution, the initial vectors $\mathbf{g}_{0}$ and $\mathbf{p}_{0}$
in CG are specified as

\begin{equation}
\mathbf{g}_{0}=\mathbf{A}^{T}(\mathbf{A}\mathbf{x}_{0}-\mathbf{y}),\label{eq: initial g}
\end{equation}

\begin{equation}
\mathbf{p}_{0}=-\mathbf{g}_{0}.\label{eq: initial p}
\end{equation}
In each iteration the tentative solution is updated by 

\begin{equation}
\mathbf{x}_{k+1}=\mathbf{x}_{k}+\alpha_{k}\mathbf{p}_{k},
\end{equation}
where $\mathbf{p}_{k}$ is a search direction with step-size $\alpha_{k}$
(defined as $\alpha$ in Step \ref{set alpha_k} of Algorithm \ref{CG-STEP}).
Part of the efficiency of CG is due to the facts that each new search
direction is conjugated with all previous directions and the step-size
is optimal for minimizing the quadratic function that is the residual
of (\ref{eq:LS2}). The directions are created according to

\begin{equation}
\mathbf{p}_{k+1}=\mathbf{-g}_{k+1}+\beta_{k}\mathbf{p}_{k}\label{eq:CGdir}
\end{equation}
(with $\beta_{k}$ defined as $\beta$ in Step \ref{set beta_k} of
Algorithm \ref{CG-STEP}), where the new gradient is given by

\begin{equation}
\mathbf{g}_{k+1}=\mathbf{A}^{T}(\mathbf{A}\mathbf{x}_{k+1}-\mathbf{y}),\label{eq: g_k+1}
\end{equation}
considering the new point $\mathbf{x}_{k+1}$, or recursively, which
is computationally more efficient, by
\begin{equation}
\mathbf{g}_{k+1}=\mathbf{g}_{k}+\alpha_{k}\mathbf{A}^{T}\mathbf{A}\mathbf{p}_{k}.\label{eq:recursive}
\end{equation}

The computation of the new direction removes part of the previous
direction from the gradient; the specified $\beta_{k}$ results in
\begin{equation}
\mathbf{p}_{k+1}^{T}\mathbf{A}^{T}\mathbf{A}\mathbf{p}_{k}=(\mathbf{-g}_{k+1}+\beta_{k}\mathbf{p}_{k})^{T}\mathbf{A}^{T}\mathbf{A}\mathbf{p}_{k}=0,
\end{equation}
which shows that the directions are conjugated.

In each iteration Step \ref{CG-call-CG-STEP} of CG calls CG-STEP$(\mathbf{A},\mathbf{x}_{k},\mathbf{p}_{k},\mathbf{g}_{k},\delta_{k},\mathbf{x}_{k+1},\mathbf{p}_{k+1},\mathbf{g}_{k+1},\delta_{k+1})$,
which means that it requires several items from the previous iteration.
Thus writing the algorithmic step of CG as $\mathbf{x}_{k+1}=\mathbf{P}(\mathbf{x}_{k})$
is an oversimplification; the vectors $\mathbf{x}_{k+1},\:\mathbf{g}_{k+1},\:\mathbf{p}_{k+1}$
depend on all three of $\mathbf{x}_{k},\:\mathbf{g}_{k},\:\mathbf{p}_{k}$.
This matters when we try to superiorize CG, for the following essential
reason: the proof of convergence of the sequence $\left(\mathbf{x}_{k}\right)_{k=0}^{\infty}$
produced by CG to $\mathbf{x}_{LS}$ involves the maintenance, as
iterations proceed, of certain mathematical relations involving $\mathbf{x}_{k},\:\mathbf{g}_{k},\:\mathbf{p}_{k}$.
It is therefore not advisable (for convergence according to the primary
criterion) to superiorize by updating $\mathbf{x}_{k}$ according
to (\ref{eq:superiorization}), without taking care of the matching
updates of $\mathbf{g}_{k}$ and $\mathbf{p}_{k}$ so that the mathematical
relations used in the proof of convergence of the sequence$\left(\mathbf{x}_{k}\right)_{k=0}^{\infty}$
are maintained.

The conjugate gradient method is generally-speaking fast, however
its convergence speed can be affected by the conditioning of the matrix
$\mathbf{A}^{T}\mathbf{A}$, being slower as the condition number
increases \cite{Vogel-2002}. In computed tomography problems, this
may mean more iterations and a very long reconstruction time. In order
to achieve a good speed of convergence, preconditioned conjugate gradient
(PCG) can be a good alternative. In this work we utilize a preconditioning
matrix $\mathbf{M}$ constructed from the filtering part of filtered
backprojection reconstruction method \cite[Chapter 8]{Herman2009}.
Each PCG iteration acts very similarly, but not exactly, to a filtered
backprojection reconstruction. In Algorithm \ref{PCG}, we describe
the standard PCG, which at each iteration calls PCG-STEP$(\mathbf{A},\mathbf{x}_{k},\mathbf{p}_{k},\mathbf{g}_{k},\delta_{k},\mathbf{x}_{k+1},\mathbf{p}_{k+1},\mathbf{g}_{k+1},\delta_{k+1},\mathbf{M})$,
as described in Algorithm \ref{PCG-STEP}.\footnote{It is not desirable for Algorithm \ref{PCG-STEP} to act exactly like
an inverse (as does filtered backprojection) because this would result
in getting from the initial point to one satisfying $f(\mathbf{x}_{k})<\varepsilon$
too quickly, not letting the superiorization to do its task of improving
the solution according to the secondary criterion.}

The computation of the new direction in PCG removes part of the previous
direction from $\mathbf{z}_{k+1}$, instead of from the gradient.
The specified $\beta_{k}$ (defined as $\beta$ in Step \ref{set P_beta_k}
of Algorithm \ref{PCG-STEP}) results in
\begin{equation}
\mathbf{p}_{k+1}^{T}\mathbf{A}^{T}\mathbf{A}\mathbf{p}_{k}=(\mathbf{-z}_{k+1}+\beta_{k}\mathbf{p}_{k})^{T}\mathbf{A}^{T}\mathbf{A}\mathbf{p}_{k}=0.
\end{equation}

In the next section we present five superiorized CG-like algorithms.
In the Appendix we discuss a proof of convergence for one such algorithm.
In Section \ref{sec:Experiments} we report on experimental results
for all five algorithms and compare such results with what can be
obtained by pure CG and also with FISTA, which is one of most efficacious
previously-published alternatives to superiorization for solving (\ref{eq:RLS2}).
We devote the rest of this section to a discussion relevant to FISTA.

As indicated in the Introduction, the NLCG approach for solving (\ref{eq:RLS2})
is based on making use of the gradient $\mathbf{g}_{k}=\mathbf{A}^{T}(\mathbf{A}\mathbf{x}_{k}-\mathbf{y})+\lambda\nabla R(\mathbf{x}_{k})$,
relying on different options, as seen in \cite{Vogel-2002,Hager2006,Dai2013,Luenberger2008}\textcolor{black}{.
However, for non-differentiable $R(\mathbf{x}_{k})$, such as TV or
the $\ell_{1}$-norm, the concept of subdifferential needs to be considered
}\cite[Chapter 8]{Rockafellar1998}\textcolor{black}{. If we substitute
any subgradient in the subdifferential for $\nabla R(\mathbf{x}_{k})$
in the previous formula for $\mathbf{g}_{k}$, then $-\mathbf{g}_{k}$
is not necessarily a descent direction and NLCG may fail to converge.}

\begin{algorithm}[!tbp]
\caption{Standard Preconditioned Conjugate Gradient $\mathrm{PCG}(\mathbf{A},\mathbf{y},\mathbf{x}_{0},\varepsilon,\mathbf{x},\mathbf{M})$}

\label{PCG}

\begin{algorithmic}[1]

\STATE{\textbf{set} $k=0$}

\STATE{\textbf{set }$\mathbf{g}_{0}=\mathbf{A}^{T}(\mathbf{A}\mathbf{x}_{0}-\mathbf{y})$}

\STATE{\textbf{set }$\mathbf{z}_{0}=\mathbf{M}\mathbf{g}_{0}$}

\STATE{\textbf{set }$\mathbf{p}_{0}=-\mathbf{z}_{0}$}

\STATE{\textbf{set }$\delta_{0}=\mathbf{g}_{0}^{T}\mathbf{z}_{0}$}

\STATE{\textbf{while }$f(\mathbf{x}_{k})>\varepsilon$ }

\STATE{~~~~\textbf{call} PCG-STEP$(\mathbf{A},\mathbf{x}_{k},\mathbf{p}_{k},\mathbf{g}_{k},\delta_{k},\mathbf{x}_{k+1},\mathbf{p}_{k+1},\mathbf{g}_{k+1},\delta_{k+1},\mathbf{M})$}

\STATE{~~~~\textbf{set $k=k+1$}} 

\STATE{\textbf{set} $\mathbf{x=x}_{k}$ } 

\end{algorithmic}
\end{algorithm}

\begin{algorithm}[!tbp]
\caption{PCG-STEP$(\mathbf{A},\mathbf{x},\mathbf{p},\mathbf{g},\delta,\mathbf{x}',\mathbf{p}',\mathbf{g}',\delta',\mathbf{M})$}

\label{PCG-STEP}

\begin{algorithmic}[1]

\STATE{\textbf{set} $\mathbf{h}=\mathbf{A}^{T}\mathbf{A}\mathbf{p}$}

\STATE{\textbf{set} $\alpha=\delta/\mathbf{p}^{T}\mathbf{h}$}

\STATE{\textbf{set }$\mathbf{x}'=\mathbf{x}+\alpha\mathbf{p}$}

\STATE{\textbf{set }$\mathbf{g}'=\mathbf{g}+\alpha\mathbf{h}$}

\STATE{\textbf{set }$\mathbf{z}'=\mathbf{M}\mathbf{g'}$}

\STATE{\textbf{set }$\delta'=\mathbf{g'}^{T}\mathbf{z'}$}

\STATE{\textbf{set }$\beta=\delta'/\delta$}\label{set P_beta_k}

\STATE{\textbf{set }$\mathbf{p}'=\mathbf{-z}'+\beta\mathbf{p}$}

\end{algorithmic}
\end{algorithm}

For this reason, proximal operators such as the ones used in (\ref{eq:ISTA})
do a better job. The FISTA algorithm \cite{Beck2009} is specified
by

\begin{equation}
\mathbf{x}_{k}=\text{prox}_{\nicefrac{\lambda}{c^{2}},R(\mathbf{x})}\left(\mathbf{u}_{k}+\frac{1}{c^{2}}\mathbf{A}^{T}(\mathbf{\mathbf{y}-A}\mathbf{u}_{k})\right),\label{eq:FISTA}
\end{equation}
where
\begin{equation}
\mathbf{u}_{k+1}=\mathbf{x}_{k}+\frac{t_{k}-1}{t_{k+1}}\left(\mathbf{x}_{k}-\mathbf{x}_{k-1}\right)\label{eq:FISTA-u}
\end{equation}
and $t_{k+1}=\left(1+\sqrt{1+4t_{k}^{2}}\right)/2$, with $t_{1}=1$
and $\mathbf{u}_{1}=\mathbf{x}_{0}$. This generates a momentum that
accelerates the ISTA algorithm specified by (\ref{eq:ISTA}), giving
to FISTA a convergence rate of order $1/k^{2}$, as compared to $1/k$
that is the order of the convergence rate of ISTA.

\section{\label{sec:Superiorization}Superiorizing the Conjugate Gradient
Method}

In this section we introduce five superiorized CG-like algorithms.
We start with a discussion of notation. Roughly speaking, if a not-superiorized
iterative algorithm produces a sequence $\left(\mathbf{x}_{k}\right)_{k=0}^{\infty}$
that converges to $\mathbf{x}_{LS}$ by using algorithmic steps of
the kind $\mathbf{x}_{k+1}=\mathbf{P}(\mathbf{x}_{k})$, then its
superiorized version works according to a formula such as (\ref{eq:superiorization}),
in which $\mathbf{v}_{k}$ is a nonascending direction for $R(\mathbf{x})$
and $\gamma_{k}$ is a step size that ensures that the next point
is superior according to the cost provided by the secondary criterion;
see, for example, \cite{Herman2012}. It has been shown that, under
reasonable assumptions, if the sequence of the $\gamma_{k}$ is summable,
then the superiorized sequence of the $\mathbf{x}_{k}$ will retain
the desirable property of the original algorithm of producing a good
approximation to $\mathbf{x}_{LS}$. In practice, the $\gamma_{k}$
are typically chosen to be a subsequence of 
\begin{equation}
\gamma_{\ell}=\gamma_{0}a^{\ell},\label{eq:gamma_sequence}
\end{equation}
where both $\gamma_{0}$ and $a$ are positive real numbers, with
$a$ strictly less than 1.

We now introduce an important clarifying notation for perturbations
of vectors in a superiorized version of an iterative algorithm. As
an example, if the perturbation is done as indicated in (\ref{eq:superiorization}),
then we use $\mathbf{x}_{k+\nicefrac{1}{2}}$ to denote $\mathbf{x}_{k}+\gamma_{k}\mathbf{v}_{k}$,
which means that (\ref{eq:superiorization}) can be replaced by $\mathbf{x}_{k+1}=\mathbf{P}(\mathbf{x}_{k+\nicefrac{1}{2}})$.
Similar notation will be adopted for other intermediate vectors in
the description of superiorized algorithms. Furthermore, for $k\geq1$,
we will use the notation $\mathbf{x}_{k-\nicefrac{1}{2}}$ to denote
$\mathbf{x}_{\left(k-1\right)+\nicefrac{1}{2}}$. 

In the superiorized algorithms described below we use the general
notation perturbed($\mathbf{x}$) to denote the perturbed version
of the vector $\mathbf{x}$. For example, if the superiorization is
done according to the discussion in the paragraph that contains (\ref{eq:superiorization}),
then
\begin{equation}
\mathrm{peturbed}\left(\mathbf{x}_{k}\right)=\mathbf{x}_{k}+\gamma_{k}\mathbf{v}_{k}.\label{eq:superiorized}
\end{equation}
This, combined with the notation introduced in the previous paragraph,
implies that $\mathbf{x}_{k+\nicefrac{1}{2}}=\mathrm{peturbed}\left(\mathbf{x}_{k}\right)$.\footnote{We note that, as is shown in \cite{Helou2016}, the proximal operator
from (\ref{eq:prox}) may be utilized to provide an alternative definition
of the perturbed vector (\ref{eq:superiorized}).}

\subsection{CG-K and Its Superiorization S-CG-K}

In our first approach to superiorizing CG, we introduce the idea of
the algorithm CG-K (where K is a fixed positive integer), whose one
iterative step is defined as K iterative steps of CG. Details of a
step CG-K$(\mathbf{A},\mathbf{y},\mathbf{x}_{0},\mathbf{x})$ of this
algorithm are provided in Algorithm \ref{CG-K}. The parameters in
the list for this algorithm CG-K$(\mathbf{A},\mathbf{y},\mathbf{x}_{0},\mathbf{x})$
are to be interpreted in the same way as for CG$(\mathbf{A},\mathbf{y},\mathbf{x}_{0},\varepsilon,\mathbf{x})$.
The only difference between CG-K and CG is in Step \ref{CG-K-termination},
which specifies the termination condition.

In each iterative step of the superiorized version of CG-K, we restart
CG using the perturbed output of the previous K iterations of CG as
the input for the next K iterations. Details of this superiorized
version are provided in Algorithm \ref{S-CG-K}. The parameter list
of the superiorized algorithm S-CG-K$(\mathbf{A},\mathbf{y},\mathbf{x}_{0},\varepsilon,\mathbf{x})$
is to be interpreted in the same way as for CG$(\mathbf{A},\mathbf{y},\mathbf{x}_{0},\varepsilon,\mathbf{x})$.
The integer K is not considered a parameter; we think of S-CG-1, S-CG-2,
etc. as different algorithms that call CG-1, CG-2, etc., respectively.

In order to see how S-CG-K$(\mathbf{A},\mathbf{y},\mathbf{x}_{0},\varepsilon,\mathbf{x})$
overcomes the previously mentioned difficulty with directly superiorizing
CG, observe that in order to obtain $\mathbf{x}_{k+1}$ from $\mathbf{x}_{k}$
(Steps \ref{S-CG-K_perturbation} and \ref{S-CG-K_call_CG-K} of Algorithm
\ref{S-CG-K}), we do not need to to worry about the maintenance of
mathematical relations with other vectors. Since CG converges to an
LS solution for any initial vector (in particular, for $\mathbf{x}_{k+\nicefrac{1}{2}}$),
$\mathbf{x}_{k+1}$ will be a good approximation to the LS solution
if K is large enough. We found that, in practice, K can be as low
as 2 or 3 for computed tomography problems. 

The S-CG-K algorithm calls the algorithm CG-K, which in turn calls
the algorithm CG-STEP that is also called in Algorithm \ref{CG} for
CG. This implies that for S-CG-K one can make use of any code or implementation
for standard CG that is available for many computer architectures
\cite{Fraysse2000,Bolz2003}.

\begin{algorithm}[!tbp]
\caption{S-CG-K$(\mathbf{A},\mathbf{y},\mathbf{x}_{0},\varepsilon,\mathbf{x})$}

\label{S-CG-K}

\begin{algorithmic}[1]

\STATE{\textbf{set} $k=0$}

\STATE{\textbf{set }$\mathbf{x}_{\nicefrac{-1}{2}}=\mathbf{x}_{0}$}

\STATE{\textbf{while }$f(\mathbf{x}_{k-\nicefrac{1}{2}})>\varepsilon$
}\label{S-CG-K_stopping}

\STATE{~~~~\textbf{set }$\mathbf{x}_{k+\nicefrac{1}{2}}=\text{perturbed}(\mathbf{x}_{k})$}\label{S-CG-K_perturbation}

\STATE{~~~~\textbf{call} CG-K$(\mathbf{A},\mathbf{y},\mathbf{x}_{k+\nicefrac{1}{2}},\mathbf{x}_{k+1})$}\label{S-CG-K_call_CG-K}

\STATE{~~~~\textbf{set $k=k+1$}} 

\STATE{\textbf{set} $\mathbf{x=x}_{k-\nicefrac{1}{2}}$ } 

\end{algorithmic}
\end{algorithm}

\medskip{}

\begin{algorithm}[!tbp]
\caption{CG-K$(\mathbf{A},\mathbf{y},\mathbf{x}_{0},\mathbf{x})$}

\label{CG-K}

\begin{algorithmic}[1]

\STATE{\textbf{set} $k=0$}

\STATE{\textbf{set }$\mathbf{g}_{0}=\mathbf{A}^{T}(\mathbf{A}\mathbf{x}_{0}-\mathbf{y})$}

\STATE{\textbf{set }$\mathbf{p}_{0}=-\mathbf{g}_{0}$}

\STATE{\textbf{set} $\delta_{0}=\left\Vert \mathbf{g}_{0}\right\Vert _{2}^{2}$}

\STATE{\textbf{while }$k\leq\mathrm{K}$ }\label{CG-K-termination}

\STATE{~~~~\textbf{call} CG-STEP$(\mathbf{A},\mathbf{x}_{k},\mathbf{p}_{k},\mathbf{g}_{k},\delta_{k},\mathbf{x}_{k+1},\mathbf{p}_{k+1},\mathbf{g}_{k+1},\delta_{k+1})$}

\STATE{~~~~\textbf{set $k=k+1$}} 

\STATE{\textbf{set} $\mathbf{x=x}_{k}$ } 

\end{algorithmic}
\end{algorithm}

\subsection{Perturbation Resilient CG for Superiorization}

Note that just repeated applications of the CG-K step (that is, using
Algorithm \ref{S-CG-K} without perturbations, meaning that $\mathbf{x}_{k+\nicefrac{1}{2}}$=$\mathbf{x}_{k}$)
does not reproduce the original CG. In order to have an algorithm
that performs exactly as CG, but is resilient to perturbations, we
need to design some alternative algorithmic steps. For this purpose
we utilized some steps proposed for NLCG \cite{Vogel-2002,Shewchuk-1994}.
Here, we present two variants of such steps\textbf{ }described in
Algorithm \ref{CG-PR-STEP} and \ref{S-CG-CD-STEP}, respectively;
the only important difference between them is the definition of $\beta$.
When combining these with perturbed superiorization steps, we obtain
the superiorized versions of two variants of CG; they are presented
as S-CG and S-CG-CD in Algorithms \ref{S-CG} and \ref{S-CG-CD},
respectively. 

An individual step of the proposed variant of the CG algorithm, as
specified in Algorithm \ref{CG-PR-STEP} (the PR is short for ``perturbation
resilient''), is not as efficient computationally as Algorithm \ref{CG-STEP}.
This is because the trick of replacing (\ref{eq: g_k+1}) by the computationally
more efficient (\ref{eq:recursive}) is mathematically incorrect when
using CG-PR-STEP and so Step \ref{CG-PR-STEP-grad} of Algorithm \ref{CG-PR-STEP},
which is similar to (\ref{eq: g_k+1}), needs to be explicitly (rather
than recursively) computed. However, repeated applications of the
CG-PR-STEP (that is, using Algorithm \ref{S-CG} without perturbations,
meaning that Step \ref{S-CG-perturbation} is replaced by $\mathbf{x}_{k+\nicefrac{1}{2}}$=$\mathbf{x}_{k}$)
results in a perturbation resilient iterative algorithm for solving
the LS problem. Thus its superiorized version (S-CG) should return
an output that is superior according to a secondary criterion (in
our case TV), while preserving performance according to the original
LS-minimization criterion.

\begin{algorithm}[!tbp]
\caption{S-CG$(\mathbf{A},\mathbf{y},\mathbf{x}_{0},\varepsilon,\mathbf{x})$}

\label{S-CG}

\begin{algorithmic}[1]

\STATE{\textbf{set }$\mathbf{\mathbf{\mathbf{x}_{\nicefrac{1}{2}}}}=\mathbf{x}_{0}$}\label{S-CG-k_0}

\STATE{\textbf{set }$\mathbf{g}_{0}=\mathbf{A}^{T}(\mathbf{A}\mathbf{x}_{0}-\mathbf{y})$}\label{initial g}

\STATE{\textbf{set }$\mathbf{p}_{0}=-\mathbf{g}_{0}$}

\STATE{\textbf{set} $\mathbf{h}_{0}=\mathbf{A}^{T}\mathbf{A}\mathbf{p}_{0}$}\label{S-CG-h_0}

\STATE{\textbf{set} $\alpha=-\mathbf{g}_{0}^{T}\mathbf{p}_{0}/\mathbf{p}_{0}^{T}\mathbf{h}_{0}$}\label{initial alpha}

\STATE{\textbf{set }$\mathbf{x}_{1}=\mathbf{x}_{0}+\alpha\mathbf{p}_{0}$}\label{update x_0}

\STATE{\textbf{set $k=1$}} 

\STATE{\textbf{while }$f(\mathbf{x}_{k-\nicefrac{1}{2}})>\varepsilon$
}\label{S-CG-stopping}

\STATE{~~~~\textbf{set }$\mathbf{x}_{k+\nicefrac{1}{2}}=\text{perturbed}(\mathbf{x}_{k})$}\label{S-CG-perturbation}

\STATE{~~~~\textbf{call} CG-PR-STEP$(\mathbf{A},\mathbf{y},\mathbf{x}_{k+\nicefrac{1}{2}},\mathbf{p}_{k-1},\mathbf{h}_{k-1},\mathbf{x}_{k+1},\mathbf{p}_{k},\mathbf{h}_{k})$}\label{S-CG-call-CG-PR-step}

\STATE{~~~~\textbf{set $k=k+1$}} \label{S-CG-update_k}

\STATE{\textbf{set} $\mathbf{x=x}_{k-\nicefrac{1}{2}}$ } 

\end{algorithmic}
\end{algorithm}

\begin{algorithm}[!tbp]
\caption{CG-PR-STEP$(\mathbf{A},\mathbf{y},\mathbf{x},\mathbf{p},\mathbf{h},\mathbf{x}',\mathbf{p}',\mathbf{h}')$}
\label{CG-PR-STEP}

\begin{algorithmic}[1]

\STATE{\textbf{set }$\mathbf{g}'=\mathbf{A}^{T}(\mathbf{A}\mathbf{x}-\mathbf{y})$}\label{CG-PR-STEP-grad}

\STATE{\textbf{set }$\beta=\mathbf{g'}^{T}\mathbf{h}/\mathbf{p}^{T}\mathbf{h}$}\label{CG-PR-STEP-beta}

\STATE{\textbf{set }$\mathbf{p}'=\mathbf{-g}'+\beta\mathbf{p}$}\label{CG-PR-STEP-new_p}

\STATE{\textbf{set} $\mathbf{h}'=\mathbf{A}^{T}\mathbf{A}\mathbf{p}'$}\label{CG-PR-STEP-h}

\STATE{\textbf{set} $\alpha=-\mathbf{g'}^{T}\mathbf{p}'/\mathbf{p'}^{T}\mathbf{h}'$}\label{CG-PR-STEP-alpha}

\STATE{\textbf{set }$\mathbf{x}'=\mathbf{x}+\alpha\mathbf{p'}$}\label{CG-PR-STEP-image_update}

\end{algorithmic}
\end{algorithm}

\begin{algorithm}[!tbp]
\caption{S-CG-CD$(\mathbf{A},\mathbf{y},\mathbf{x}_{0},\varepsilon,\mathbf{x})$}

\label{S-CG-CD}

\begin{algorithmic}[1]

\STATE{\textbf{set }$\mathbf{\mathbf{\mathbf{x}_{\nicefrac{1}{2}}}}=\mathbf{x}_{0}$}

\STATE{\textbf{set }$\mathbf{g}_{0}=\mathbf{A}^{T}(\mathbf{A}\mathbf{x}_{0}-\mathbf{y})$}

\STATE{\textbf{set }$\mathbf{p}_{0}=-\mathbf{g}_{0}$}

\STATE{\textbf{set} $\mathbf{h}_{0}=\mathbf{A}^{T}\mathbf{A}\mathbf{p}_{0}$}

\STATE{\textbf{set} $\alpha=-\mathbf{g}_{0}^{T}\mathbf{p}_{0}/\mathbf{p}_{0}^{T}\mathbf{h}_{0}$}

\STATE{\textbf{set }$\mathbf{x}_{1}=\mathbf{x}_{0}+\alpha\mathbf{p}_{0}$}

\STATE{\textbf{set $k=1$}} 

\STATE{\textbf{while }$f(\mathbf{x}_{k-\nicefrac{1}{2}})>\varepsilon$
}

\STATE{~~~~\textbf{set }$\mathbf{x}_{k+\nicefrac{1}{2}}=\text{perturbed}(\mathbf{x}_{k})$}

\STATE{~~~~\textbf{call} CG-CD-STEP$(\mathbf{A},\mathbf{y},\mathbf{x}_{k+\nicefrac{1}{2}},\mathbf{p}_{k-1},\mathbf{g}_{k-1},\mathbf{x}_{k+1},\mathbf{p}_{k},\mathbf{g}_{k})$}

\STATE{~~~~\textbf{set $k=k+1$}} 

\STATE{\textbf{set} $\mathbf{x=x}_{k-\nicefrac{1}{2}}$ } 

\end{algorithmic}
\end{algorithm}

\begin{algorithm}[!tbp]
\caption{CG-CD-STEP$(\mathbf{A},\mathbf{y},\mathbf{x},\mathbf{p},\mathbf{g},\mathbf{x'},\mathbf{p'},\mathbf{g'})$ }
\label{S-CG-CD-STEP}

\begin{algorithmic}[1]

\STATE{\textbf{set }$\mathbf{g'}=\mathbf{A}^{T}(\mathbf{A}\mathbf{x}-\mathbf{y})$}

\STATE{\textbf{set }$\beta=-\left\Vert \mathbf{g'}\right\Vert _{2}^{2}/\mathbf{g}^{T}\mathbf{p}$}\label{S-CG-CD-STEP_beta}

\STATE{\textbf{set }$\mathbf{p'}=\mathbf{-g'}+\beta\mathbf{p}$}

\STATE{\textbf{set} $\mathbf{h}'=\mathbf{A}^{T}\mathbf{A}\mathbf{p'}$}

\STATE{\textbf{set} $\alpha=-\mathbf{g'}^{T}\mathbf{p'}/\mathbf{p'}^{T}\mathbf{h'}$}

\STATE{\textbf{set }$\mathbf{x'}=\mathbf{x}+\alpha\mathbf{p'}$}

\end{algorithmic}
\end{algorithm}

Furthermore, different definitions of $\beta$ can be utilized \cite{Hager2006,Dai2013}.
One of these, namely $\beta=-||\mathbf{g'}||_{2}^{2}/\mathbf{g}^{T}\mathbf{p}$,
is very effective; we use it in Step \ref{S-CG-CD-STEP_beta} of Algorithm
\ref{S-CG-CD-STEP}. (It has been referred to in the literature as
the ``conjugate descent'' rule, hence the ``CD'' in the name of
the algorithm.) This choice of $\beta$ provides a smooth behavior
of the superiorized CG algorithm and it resulted in the best reconstruction
quality among all our experimentally-tested algorithms.

\subsection{PCG-K and Its Superiorization S-PCG-K}

As is the case with standard CG, the PCG algorithm cannot be directly
superiorized. However, the same modifications that we proposed for
CG can be done also to PCG. The first of these is superiorized PCG-K
(S-PCG-K). The operator for each iteration of S-PCG-K (Algorithm \ref{S-PCG-K})
is specified by PCG-K (Algorithm \ref{PCG-K}), with K being the number
of iterations of preconditioned CG, similarly to what was defined
for CG-K. In practice, a small number K is sufficient for S-PCG-K
to achieve good results. (In fact, a large value of K may be undesirable.
This is because, the output of such PCG-K would be very close to the
minimum of the least squares, defined in (\ref{eq:LS1}), resulting
in early violation of the condition in Step \ref{S-PCG-K_stopping}
of Algorithm \ref{S-PCG-K}. This would prevent the execution of the
number of perturbations by Step \ref{S-PCG-K_perturbation} of Algorithm
\ref{S-PCG-K} that is needed to get to a noticeably superior output
according to the secondary criterion.) The not-superiorized version
of S-PCG-K (that is, using Algorithm \ref{S-PCG-K} without perturbations,
meaning that Step \ref{S-PCG-K_perturbation} is replaced by $\mathbf{x}_{k+\nicefrac{1}{2}}$=$\mathbf{x}_{k}$)
is a nonexpansive algorithm and, consequently, it is strongly perturbation
resilient \cite{Herman2012}.

\subsection{Perturbation Resilient PCG for Superiorization}

Since repeated application of PCG-K does not mimic the PCG steps exactly,
a second variant of preconditioned CG was proposed for superiorization;
this new superiorized algorithm is provided as Algorithm \ref{S-PCG}.
As was done for CG, the standard version of preconditioned CG was
modified so that it becomes perturbation resilient. This way the algorithm
will eventually converge to the LS solution, considering that bounded
and summable perturbations are applied \cite{Herman2012}. This algorithm
is named S-PCG.

\begin{algorithm}[!tbp]
\caption{S-PCG-K$(\mathbf{A},\mathbf{y},\mathbf{x}_{0},\varepsilon,\mathbf{x},\mathbf{M})$}

\label{S-PCG-K}

\begin{algorithmic}[1]

\STATE{\textbf{set} $k=0$}

\STATE{\textbf{set }$\mathbf{x}_{\nicefrac{-1}{2}}=\mathbf{x}_{0}$}

\STATE{\textbf{while }$f(\mathbf{x}_{k\nicefrac{-1}{2}})>\varepsilon$
}\label{S-PCG-K_stopping}

\STATE{~~~~\textbf{set }$\mathbf{x}_{k+\nicefrac{1}{2}}=\text{perturbed}(\mathbf{x}_{k})$}\label{S-PCG-K_perturbation}

\STATE{~~~~\textbf{call} PCG-K$(\mathbf{A},\mathbf{y},\mathbf{x}_{k+\nicefrac{1}{2}},\mathbf{x}_{k+1},\mathbf{M})$}

\STATE{~~~~\textbf{set $k=k+1$}} 

\STATE{\textbf{set} $\mathbf{x=x}_{k\nicefrac{-1}{2}}$ } 

\end{algorithmic}
\end{algorithm}

\begin{algorithm}[!tbp]
\caption{PCG-K$(\mathbf{A},\mathbf{y},\mathbf{x}_{0},\mathbf{x},\mathbf{M})$}

\label{PCG-K}

\begin{algorithmic}[1]

\STATE{\textbf{set} $k=0$}

\STATE{\textbf{set }$\mathbf{g}_{0}=\mathbf{A}^{T}(\mathbf{A}\mathbf{x}_{0}-\mathbf{y})$}

\STATE{\textbf{set }$\mathbf{z}_{0}=\mathbf{M}\mathbf{g}_{0}$}\label{PCG-K-preconditioning}

\STATE{\textbf{set }$\mathbf{p}_{0}=-\mathbf{z}_{0}$}

\STATE{\textbf{set }$\delta_{0}=\mathbf{g}_{0}^{T}\mathbf{z}_{0}$}

\STATE{\textbf{while }$k\leq\mathrm{K}$ }

\STATE{~~~~\textbf{call} PCG-STEP$(\mathbf{A},\mathbf{x}_{k},\mathbf{p}_{k},\mathbf{g}_{k},\delta_{k},\mathbf{x}_{k+1},\mathbf{p}_{k+1},\mathbf{g}_{k+1},\delta_{k+1},\mathbf{M})$}

\STATE{~~~~\textbf{set $k=k+1$}} 

\STATE{\textbf{set} $\mathbf{x=x}_{k}$ } 

\end{algorithmic}
\end{algorithm}

\begin{algorithm}[!tbp]
\caption{S-PCG$(\mathbf{A},\mathbf{y},\mathbf{x}_{0},\varepsilon,\mathbf{x},\mathbf{M})$}

\label{S-PCG}

\begin{algorithmic}[1]

\STATE{\textbf{set }$\mathbf{\mathbf{\mathbf{x}_{\nicefrac{1}{2}}}}=\mathbf{x}_{0}$}

\STATE{\textbf{set }$\mathbf{g}_{0}=\mathbf{A}^{T}(\mathbf{A}\mathbf{x}_{0}-\mathbf{y})$}

\STATE{\textbf{set }$\mathbf{z}_{0}=\mathbf{M}\mathbf{g}_{0}$}\label{S-PCG_preconditioning}

\STATE{\textbf{set }$\mathbf{p}_{0}=-\mathbf{z}_{0}$}

\STATE{\textbf{set} $\mathbf{h}_{0}=\mathbf{A}^{T}\mathbf{A}\mathbf{p}_{0}$}

\STATE{\textbf{set} $\alpha=-\mathbf{g}_{0}^{T}\mathbf{p}_{0}/\mathbf{p}_{0}^{T}\mathbf{h}_{0}$}

\STATE{\textbf{set }$\mathbf{x}_{1}=\mathbf{x}_{0}+\alpha\mathbf{p}_{0}$}

\STATE{\textbf{set $k=1$}} 

\STATE{\textbf{while }$f(\mathbf{x}_{k-\nicefrac{1}{2}})>\varepsilon$
}

\STATE{~~~~\textbf{set }$\mathbf{x}_{k+\nicefrac{1}{2}}=\text{perturbed}(\mathbf{x}_{k})$}

\STATE{~~~~\textbf{call} PCG-PR-STEP$(\mathbf{A},\mathbf{y},\mathbf{x}_{k+\nicefrac{1}{2}},\mathbf{p}_{k-1},\mathbf{h}_{k-1},\mathbf{x}_{k+1},\mathbf{p}_{k},\mathbf{h}_{k},\mathbf{M})$}

\STATE{~~~~\textbf{set $k=k+1$}} 

\STATE{\textbf{set} $\mathbf{x=x}_{k-\nicefrac{1}{2}}$ } 

\end{algorithmic}
\end{algorithm}

\begin{algorithm}[!tbp]
\caption{PCG-PR-STEP$(\mathbf{A},\mathbf{y},\mathbf{x},\mathbf{p},\mathbf{h},\mathbf{x'},\mathbf{p'},\mathbf{h'},\mathbf{M})$ }
\label{S-PCG-PR-STEP}

\begin{algorithmic}[1]

\STATE{\textbf{set }$\mathbf{g'}=\mathbf{A}^{T}(\mathbf{A}\mathbf{x}-\mathbf{y})$}

\STATE{\textbf{set }$\mathbf{z'}=\mathbf{M}\mathbf{g'}$}\label{S-PCG-PR-STEP_preconditioning}

\STATE{\textbf{set }$\beta=\mathbf{z'}^{T}\mathbf{h}/\mathbf{p}^{T}\mathbf{h}$}

\STATE{\textbf{set }$\mathbf{p'}=\mathbf{-z'}+\beta\mathbf{p}$}

\STATE{\textbf{set} $\mathbf{h'}=\mathbf{A}^{T}\mathbf{A}\mathbf{p'}$}

\STATE{\textbf{set} $\alpha=-\mathbf{g'}^{T}\mathbf{p'}/\mathbf{p'}^{T}\mathbf{h'}$}

\STATE{\textbf{set }$\mathbf{x'}=\mathbf{x}+\alpha\mathbf{p'}$}

\end{algorithmic}
\end{algorithm}

\subsection{Filters for Preconditioning }

In order to specify more precisely the details of S-PCG and S-PCG-K,
we complete this section with a discussion of the preconditioning
filter $\mathbf{M}$ used in Step \ref{PCG-K-preconditioning} of
Algorithm \ref{PCG-K}, Step \ref{S-PCG_preconditioning} of Algorithm
\ref{S-PCG} and Step \ref{S-PCG-PR-STEP_preconditioning} of Algorithm
\ref{S-PCG-PR-STEP}. It is a filter applied in the frequency domain,
in the same way as many filtering operations utilized in the filtered
back projection method \cite[Chapter 8]{Herman2009}. It can be defined
as

\begin{equation}
\begin{array}{c}
\mathbf{z}=\mathbf{A}^{T}\mathbf{F}^{T}\mathbf{C}\mathbf{F}(\mathbf{A}\mathbf{x}-\mathbf{y}),\end{array}
\end{equation}
where $\mathbf{F}$ represents a set of 1D Fourier transforms applied
to each projection and $\mathbf{C}$ is a diagonal matrix with the
filter factors in the frequency domain for each projection information
(in Fourier domain). The filter factors can be written as

\begin{equation}
c(\omega)=(|\omega|+\mu)\times(0.54+0.46\,cos(\omega)),\label{eq:filter_factors_pre}
\end{equation}
where $\omega$ is the frequency, $\mu$ is small number to avoid
any singularity in the preconditioning matrix, and $(0.54+0.46\,cos(\omega))$
is the Hamming window, to reduce the gain in high frequencies that
may amplify noise too much. The Hamming window is commonly used to
``regularize'' the filter factors in filtered backprojection \cite{Herman2009}.

\section{\label{sec:Experiments}Experimental Results}

In these experiments we utilized a parallel ray tomographic simulation
with 256 angles, uniformly spaced between 0 and 180 degrees, and 512
rays/angle. Gaussian noise with zero mean and $\sigma^{2}$ variance
was added to the synthetically generated data, where $\sigma^{2}$
was computed to provide 5\% of noise, or SNR of approximately 26dB.
We reconstruct the images with 512$\times$512 pixels, using the following
methods:
\begin{itemize}
\item CG: Algorithm \ref{CG}, the conjugate gradient method for solving
problem (\ref{eq:LS1});
\item S-CG-2: Algorithm \ref{S-CG-K} with $K$=2, the superiorized CG algorithm
using the operator CG-2 of Algorithm \ref{CG-K};
\item S-CG: Algorithm \ref{S-CG}, the superiorized CG algorithm using the
perturbation resilient CG operator of Algorithm \ref{CG-PR-STEP};
\item S-CG-CD: Algorithm \ref{S-CG-CD}, variant of S-CG using the CG-CD
operator of Algorithm \ref{S-CG-CD-STEP};
\item S-PCG-2: Algorithm \ref{S-PCG-K} with $K$=2, the superiorized preconditioned
CG algorithm using the operator PCG-2 of Algorithm \ref{PCG-K};
\item S-PCG: Algorithm \ref{S-PCG}; the superiorized preconditioned CG
algorithm using the operator PCG-PR-STEP of Algorithm \ref{S-PCG-PR-STEP};
\item FISTA: a reference method for the TV-regularized LS problem, specified
by (\ref{eq:FISTA}).
\end{itemize}
For all the algorithms, we selected the initial guess at the solution,
$\mathbf{x}_{0}$, to be the $J$-dimensional vector all of whose
components are zero. The $\varepsilon$ value for the stopping criterion
(as in Step \ref{sc} of Algorithm \ref{CG} and Step \ref{S-CG-K_stopping}
of Algorithm \ref{S-CG-K}, etc.) is defined as $\varepsilon=N\sigma^{2}$
where $\sigma^{2}$ is the noise variance and $N=256\times512$ is
the size of the data vector. For all superiorized algorithms, the
constant $a$, in (\ref{eq:gamma_sequence}), was set to .975 and
$\gamma_{0}$ was manually selected to achieve the best performance
(it is different for each method). Algorithms for all these methods
(except for FISTA) are specified in the previous two sections and,
in each case, the output $\mathbf{x}$ of the algorithm is provided
by its last step. The following paragraphs explain the correct interpretation
of the results reported in Figures \ref{fig:ERROR}, \ref{fig:COST}
and \ref{fig:TV_NORM}. 

In Figure \ref{fig:ERROR} we compare algorithm performance in terms
of quality and reconstruction speed, as judged by reconstruction error
versus time. Each open circle on the curves corresponds to an iteration.
The red boxes indicate the iteration when the stopping criterion is
first satisfied. The output images are shown in Figure \ref{fig:images},
in which the $k$ used to get the output is also indicated. For illustration
purposes, we show error values even beyond the stopping criterion.
We see in Figure \ref{fig:ERROR} that the algorithm S-CG-CD has the
lowest error among all reported algorithms and that S-PCG-2 is the
superiorized algorithm that satisfied the stopping criterion in the
least amount of time.

\begin{figure}[!tbp]
\includegraphics[scale=0.54]{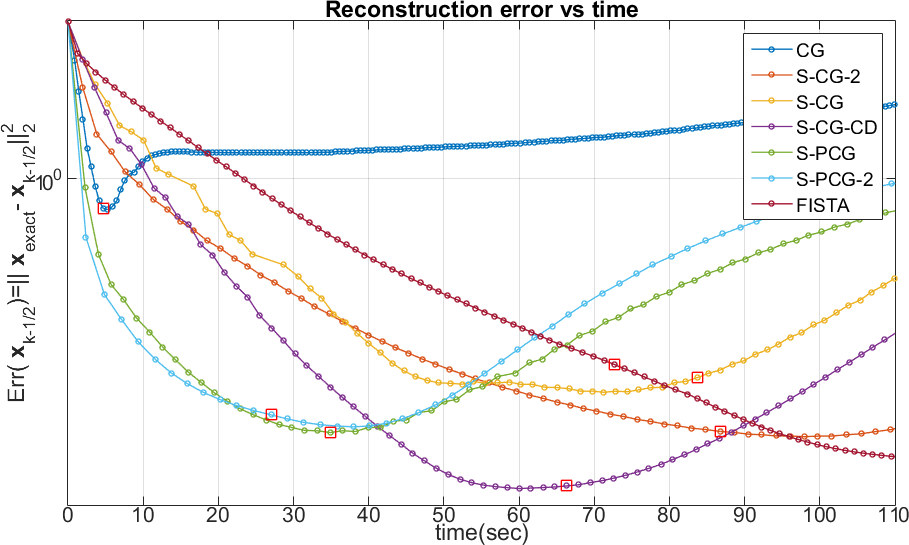}\caption{Reconstruction error versus time}
\label{fig:ERROR}
\end{figure}

Figure \ref{fig:COST} reports on the half-squared-residual $f$ of
the reconstructed images over time. Again we report beyond the iteration
provided by the stopping criterion, with the purpose to demonstrate
the behavior of superiorization algorithms, which converge to the
(not regularized) least squares solution. On the other hand, regularized
algorithms, such as FISTA, stay nearly at the same value of $f$ as
the iterations go beyond the stopping criterion.

\begin{figure}[!tbp]
\includegraphics[scale=0.54]{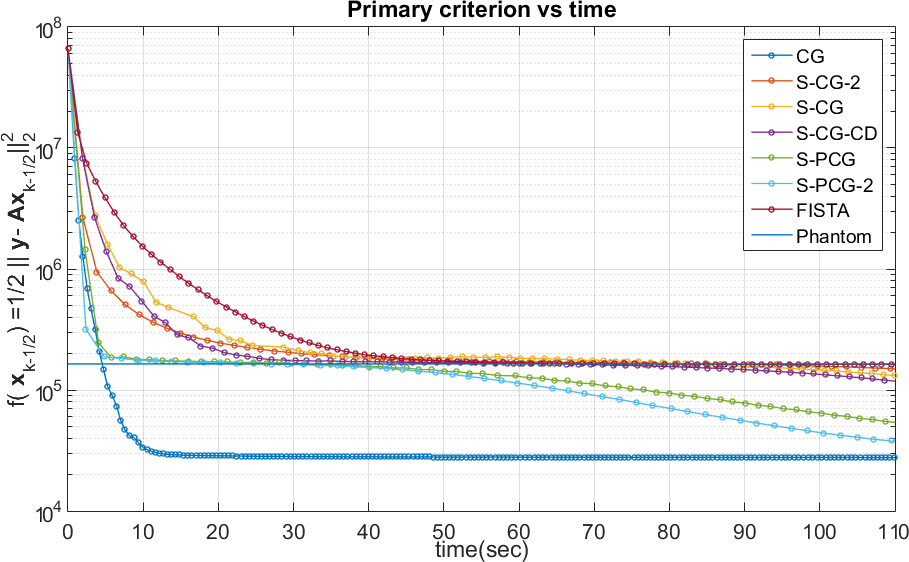}\caption{Least squares residual versus time}
\label{fig:COST}
\end{figure}

In Figure \ref{fig:TV_NORM}, we show how the smoothness, as measured
by the TV norm of the reconstructed images, behaves over time. Starting
with a zero-valued image $\mathbf{x}_{0}$ (whose TV norm is 0), we
see that the reconstructed images rapidly reach having a TV norm very
close to the TV norm of the phantom. After a while the TV norm of
the superiorization algorithms grows again, due to the reduction of
the $\gamma_{k}$ values in the perturbations (\ref{eq:superiorized});
but this is not relevant as far as the output is concerned, since
it happens after the stopping criterion is satisfied.

\begin{figure}[!tbp]
\includegraphics[scale=0.54]{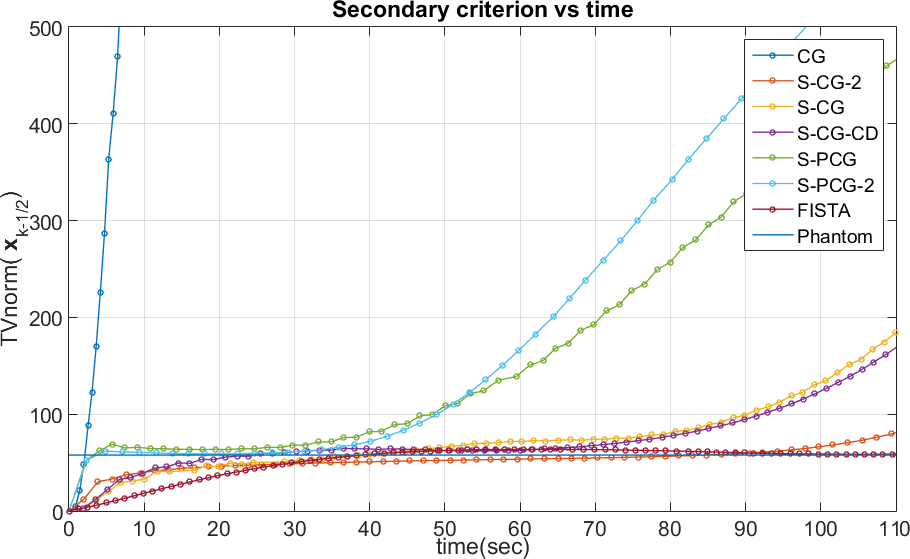}

\caption{TV norm versus time}
\label{fig:TV_NORM}
\end{figure}

Figure 4 presents the output reconstructions produced by the tested
methods, indicating the number of iterations required to reach those
results. As a reference, a reconstruction produced by a filtered backprojection
method \cite[Chapter 8]{Herman2009} is also shown. These results
can be used for subjective visual evaluation.

\begin{figure}[!tbp]
\begin{minipage}[t]{0.34\columnwidth}%
\subfloat[Filtered backprojection]{\includegraphics[scale=0.23]{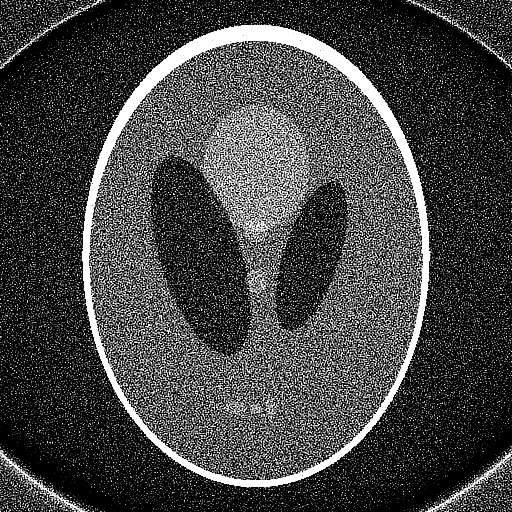}

}%
\end{minipage}%
\begin{minipage}[t]{0.34\columnwidth}%
\subfloat[CG, iteration 9]{\includegraphics[scale=0.23]{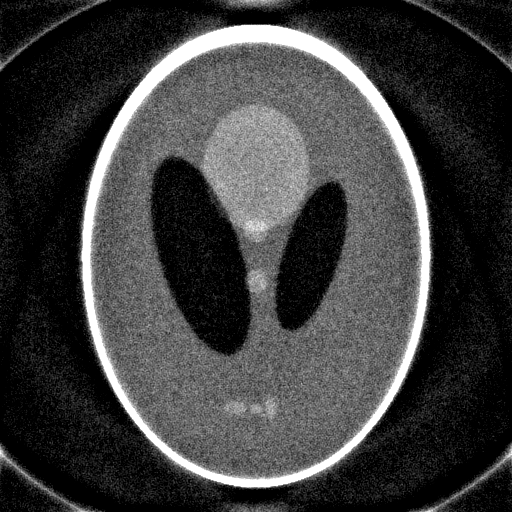}

}%
\end{minipage}%
\begin{minipage}[t]{0.34\columnwidth}%
\subfloat[S-CG-2, iteration 50]{\includegraphics[scale=0.23]{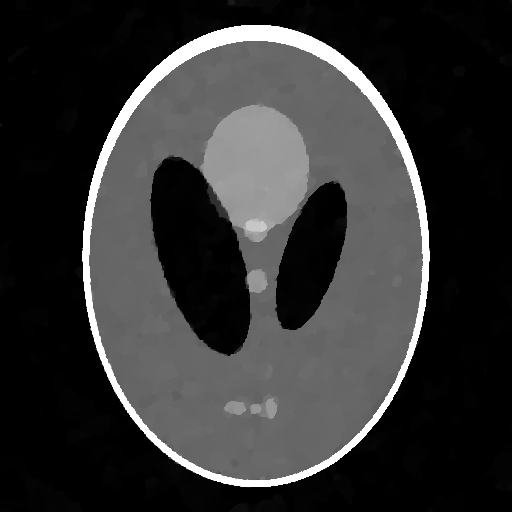}

}%
\end{minipage}

\begin{minipage}[t]{0.34\columnwidth}%
\subfloat[S-CG, iteration 51]{\includegraphics[scale=0.23]{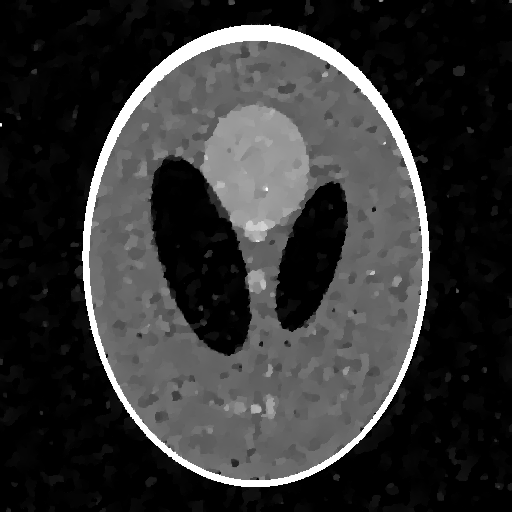}

}%
\end{minipage}%
\begin{minipage}[t]{0.34\columnwidth}%
\subfloat[S-CG-CD, iteration 43]{\includegraphics[scale=0.23]{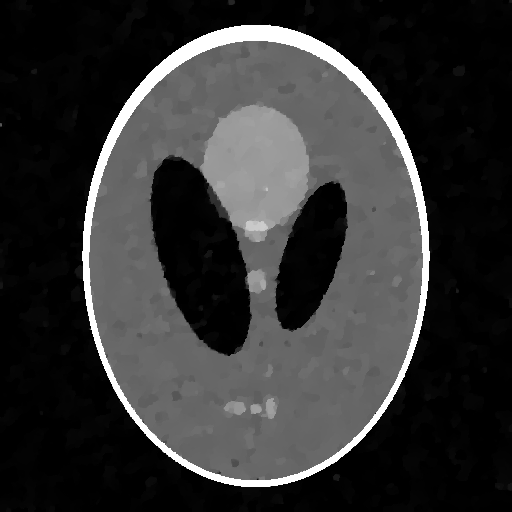}

}%
\end{minipage}%
\begin{minipage}[t]{0.34\columnwidth}%
\subfloat[S-PCG, iteration 20]{\includegraphics[scale=0.23]{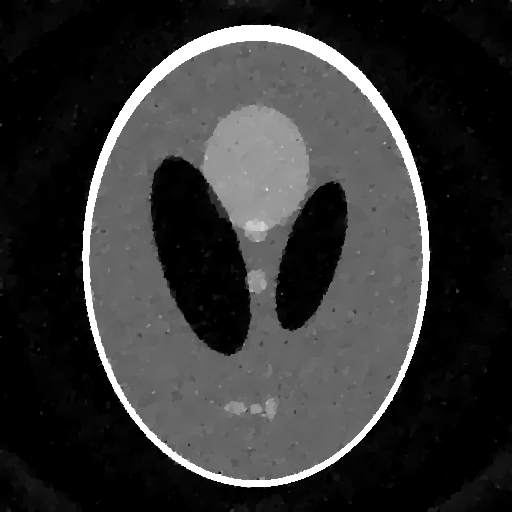}

}%
\end{minipage}

\begin{minipage}[t]{0.34\columnwidth}%
\subfloat[S-PCG-2, iteration 13]{\includegraphics[scale=0.23]{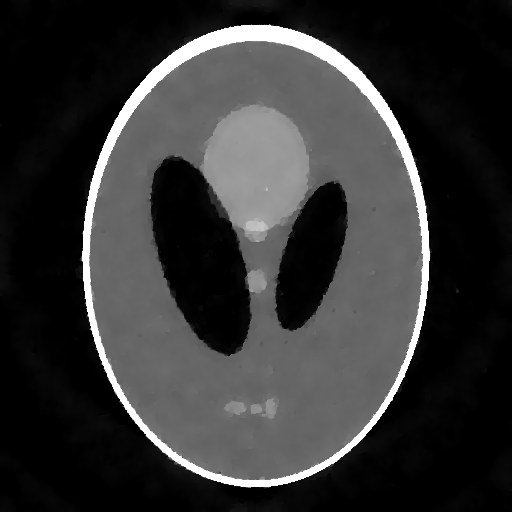}

}%
\end{minipage}%
\begin{minipage}[t]{0.34\columnwidth}%
\subfloat[FISTA, iteration 61]{\includegraphics[scale=0.23]{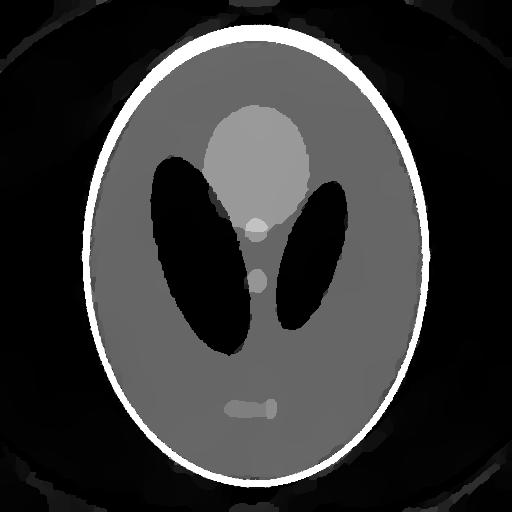}

}%
\end{minipage}%
\begin{minipage}[t]{0.34\columnwidth}%
\subfloat[Phantom]{\includegraphics[scale=0.23]{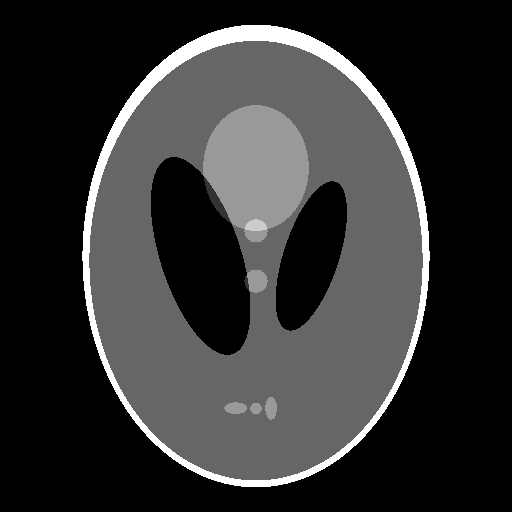}}%
\end{minipage}\caption{Visual results}
\label{fig:images}
\end{figure}

\section{Discussion and Conclusions\label{sec:Conclusions}}

In this paper we have applied the superiorization approach to the
conjugate gradient method, as well as to its preconditioned version.
The total variation norm, commonly utilized as a regularizing penalty,
was utilized as the secondary criterion for superiorization. In the
Appendix we prove for one of the superiorized versions of the conjugate
gradient method (namely, S-CG) that it behaves in the manner that
is expected for a superiorized algorithm (roughly meaning that it
gets as good results as the unsuperiorized version according to the
primary criterion, but with improvements according to the secondary
criterion; see, for example, \cite{Herman2012} for a more precise
discussion). The experimental results have illustrated that the proposed
algorithms are fast, achieving as good results as algorithms such
as FISTA, but in less time; see the plots for S-CG-CD, S-PCG and S-PCG-K
in Figure \ref{fig:ERROR}. The proposed method is a practical solution
for the image reconstruction problem when good reconstruction quality
and low reconstruction time are both important requirements for the
desired reconstruction algorithm.

\section*{Acknowledgments{\normalsize{}\label{sec:Acknowledgments}}}

The work by the first author was partially supported by CNPq grant
475553/2013-6. The work by the second author was supported in part
by China Scholarship Council. The authors thank Elias Salomão Helou
Neto for discussions on earlier versions of this paper.

\section*{Appendix: Proof that Superiorized CG (Algorithm \ref{S-CG}) Terminates\label{sec:Appendix:-convergence-proof}}

The specifications of the parameters of Algorithm \ref{S-CG} are: 
\begin{itemize}
\item $J$ is a positive integer (the dimensionality of the output vector
of the algorithm).
\item $L$ is a positive integer (the number of data elements in the sinogram).
\item $\mathbf{A}$ is an $L\times J$ (system) matrix. 
\item $\mathbf{y}$ is an $L$-dimensional (measurement) vector.
\item $\varepsilon$ is a positive number (used for algorithm termination).
\item $\mathbf{x}_{0}$ is a $J$-dimensional (input) vector (the initial
guess vector).
\item $\mathbf{x}$ is a $J$-dimensional (output) vector.
\end{itemize}
The quadratic objective function for the data fidelity constraints
(the primary criterion) is the half-squared-residual

\begin{eqnarray}
f(\mathbf{x}) & = & \frac{1}{2}\left\Vert |\mathbf{y-Ax}\right\Vert _{2}^{2}.\label{eq:fuai_x}
\end{eqnarray}
We define 
\begin{equation}
\varepsilon_{0}=\min_{\mathbf{x}}\frac{1}{2}||\mathbf{y-Ax}||_{2}^{2}.\label{eq:epsilon_0}
\end{equation}
This implies that $\varepsilon_{0}=f\left(\mathbf{x}_{LS}\right)$;
see (\ref{eq:LS1}). We note that termination of Algorithm \ref{S-CG}
can happen only if the condition in Step \ref{S-CG-stopping} of the
algorithm is violated. That means that $f\left(\mathbf{x}_{k-\nicefrac{1}{2}}\right)\leq\varepsilon$
and also that the output $\mathbf{x}$ of the algorithm satisfies
$f\left(\mathbf{x}\right)\leq\varepsilon$.
\begin{thm}
\label{Thm: Conv-1}Given any positive number $\varepsilon$ such
that $\varepsilon>\varepsilon_{0}$, Algorithm \ref{S-CG} terminates.
\end{thm}
We prove this theorem through a sequence of intermediate results.
We note the fact that if Step \ref{S-CG-call-CG-PR-step} of Algorithm
\ref{S-CG} is executed, it has to be the case that at that time $k\geq1$.
Also, we consider what happens during the execution of Step \ref{S-CG-call-CG-PR-step}
of Algorithm \ref{S-CG}, which is a call to Algorithm \ref{CG-PR-STEP}:
After executing Step \ref{CG-PR-STEP-grad}, all occurrences of $\mathbf{g'}$
in Algorithm \ref{CG-PR-STEP} may also be denoted as $\mathbf{g}_{k+\nicefrac{1}{2}}$,
where
\begin{equation}
\mathbf{g}_{k+\nicefrac{1}{2}}=\mathbf{A}^{T}\left(\mathbf{A}\mathbf{x}_{k+\nicefrac{1}{2}}-\mathbf{y}\right),\label{g_z,k}
\end{equation}
which is consistent with (\ref{eq: g_k+1}) and is notation that we
use from now on.
\begin{prop}
\label{prop: ortho gp}Upon the execution of Step \ref{S-CG-call-CG-PR-step}
of Algorithm \ref{S-CG}, it is the case that, for all integers $n$
such that $0\leq n\leq k$,

\begin{equation}
\left\langle \mathbf{g}_{n+1},\mathbf{p}_{n}\right\rangle =0,\label{eq: <g, p=003009=00003D0}
\end{equation}
where, $\mathbf{g}_{n+1}$ is defined by (\ref{eq: g_k+1}) with the
$k$ there replaced by $n$.\end{prop}
\begin{proof}
We first prove that (\ref{eq: <g, p=003009=00003D0}) holds for $n=0$.
It follows from (\ref{eq: g_k+1}) and Steps \ref{initial g} and
\ref{update x_0} of Algorithm \ref{S-CG} that

\[
\mathbf{g}_{1}=\mathbf{g}_{0}+\alpha\mathbf{A}^{T}\mathbf{A}\mathbf{p}_{0}.
\]

Combining this with Steps \ref{S-CG-h_0} and \ref{initial alpha}
of Algorithm \ref{S-CG}, we have that
\begin{eqnarray}
\left\langle \mathbf{g}_{1},\mathbf{p}_{0}\right\rangle  & = & \left\langle \mathbf{g}_{0},\mathbf{p}_{0}\right\rangle +\alpha\left\langle \mathbf{A}^{T}\mathbf{A}\mathbf{p}_{0},\mathbf{p}_{0}\right\rangle \nonumber \\
 & = & \left\langle \mathbf{g}_{0},\mathbf{p}_{0}\right\rangle -\frac{\left\langle \mathbf{g}_{0},\mathbf{p}_{0}\right\rangle }{\left\langle \mathbf{A}\mathbf{p}_{0},\mathbf{A}\mathbf{p}_{0}\right\rangle }\left\langle \mathbf{A}\mathbf{p}_{0},\mathbf{A}\mathbf{p}_{0}\right\rangle \label{eq:initial g.p=00003D0}\\
 & = & 0.\nonumber 
\end{eqnarray}

Note that (\ref{eq:initial g.p=00003D0}) always holds after Step
\ref{update x_0} of Algorithm \ref{S-CG} since $\mathbf{g}_{1}$
and $\mathbf{p}_{0}$ do not get changed after their initial assignments.

Next we prove that, for $k\geq1$, upon the execution of Step \ref{S-CG-call-CG-PR-step}
of Algorithm \ref{S-CG}, 

\begin{equation}
\left\langle \mathbf{g}_{k+1},\mathbf{p}_{k}\right\rangle =0.\label{eq: <g_k+1, p_k=003009=00003D0}
\end{equation}

It follows from (\ref{eq: g_k+1}), (\ref{g_z,k}) and Step \ref{CG-PR-STEP-image_update}
of Algorithm \ref{CG-PR-STEP} that, upon the execution of Step \ref{S-CG-call-CG-PR-step}
of Algorithm \ref{S-CG}, 

\begin{equation}
\mathbf{g}_{k+1}=\mathbf{g}_{k+\nicefrac{1}{2}}+\alpha\mathbf{A}^{T}\mathbf{A}\mathbf{p}_{k}.\label{eq: g_k+1 iterate}
\end{equation}

Combining this with Steps \ref{CG-PR-STEP-h} and \ref{CG-PR-STEP-alpha}
of Algorithm \ref{CG-PR-STEP}, we have that
\begin{eqnarray}
\left\langle \mathbf{g}_{k+1},\mathbf{p}_{k}\right\rangle  & = & \left\langle \mathbf{g}_{k+\nicefrac{1}{2}},\mathbf{p}_{k}\right\rangle +\alpha\left\langle \mathbf{A}^{T}\mathbf{A}\mathbf{p}_{k},\mathbf{p}_{k}\right\rangle \nonumber \\
 & = & \left\langle \mathbf{g}_{k+\nicefrac{1}{2}},\mathbf{p}_{k}\right\rangle -\frac{\left\langle \mathbf{g}_{k+\nicefrac{1}{2}},\mathbf{p}_{k}\right\rangle }{\left\langle \mathbf{A}\mathbf{p}_{k},\mathbf{A}\mathbf{p}_{k}\right\rangle }\left\langle \mathbf{A}\mathbf{p}_{k},\mathbf{A}\mathbf{p}_{k}\right\rangle \label{eq:g.p=00003D0}\\
 & = & 0.\nonumber 
\end{eqnarray}

We note that, for $k\geq1$, $\mathbf{x}_{k+1}$ and $\mathbf{p}_{k}$
do not get changed once they are obtained by an execution of Step
\ref{S-CG-call-CG-PR-step} of Algorithm \ref{S-CG}. Combing this
fact and (\ref{eq: g_k+1}), we know that, for $k\geq1$, $\mathbf{g}_{k+1}$
does not get changed once it is obtained. From these facts and (\ref{eq:initial g.p=00003D0}),
(\ref{eq:g.p=00003D0}), it follows that, upon the execution of Step
\ref{S-CG-call-CG-PR-step} of Algorithm \ref{S-CG}, it is the case
that, for all integers $n$ such that $0\leq n\leq k$, (\ref{eq: <g, p=003009=00003D0})
holds.\end{proof}
\begin{prop}
\label{prop: h_k}Just before an execution of Step \ref{S-CG-call-CG-PR-step}
of Algorithm \ref{S-CG}, it holds that
\begin{equation}
\mathbf{h}_{k-1}=\mathbf{A}^{T}\mathbf{A}\mathbf{p}_{k-1}.\label{eq: h_k-1}
\end{equation}
\end{prop}
\begin{proof}
From Step \ref{S-CG-h_0} of Algorithm \ref{S-CG}, we know that,
for $k=1$, (\ref{eq: h_k-1}) holds just before the execution of
Step \ref{S-CG-call-CG-PR-step} of Algorithm \ref{S-CG}. By Step
\ref{CG-PR-STEP-h} of Algorithm \ref{CG-PR-STEP}, we see that upon
an execution of Step \ref{S-CG-call-CG-PR-step} of Algorithm \ref{S-CG},
it holds that 

\begin{equation}
\mathbf{h}_{k}=\mathbf{A}^{T}\mathbf{A}\mathbf{p}_{k}.\label{eq: h_k}
\end{equation}

These facts and the repeated executions of the \textbf{while} loop
of Algorithm \ref{S-CG} imply that, for $k\geq1$, (\ref{eq: h_k-1})
holds just before the execution of Step \ref{S-CG-call-CG-PR-step}
of Algorithm \ref{S-CG}.\end{proof}
\begin{prop}
\label{prop: otrho conj}Upon the execution of Step \ref{S-CG-call-CG-PR-step}
of Algorithm \ref{S-CG}, it is the case that

\begin{equation}
\left\langle \mathbf{A}\mathbf{p}_{k},\mathbf{A}\mathbf{p}_{k-1}\right\rangle =0.\label{eq: <p'_k, p'_k-1>=00003D0}
\end{equation}
\end{prop}
\begin{proof}
By Proposition \ref{prop: h_k}, we know that (\ref{eq: h_k-1}) holds
just before the execution of Step \ref{S-CG-call-CG-PR-step} of Algorithm
\ref{S-CG}. Consider now the execution of Step \ref{S-CG-call-CG-PR-step}
of Algorithm \ref{S-CG}, which is a call to Algorithm \ref{CG-PR-STEP}.
It is easily shown that, upon the execution of Step \ref{CG-PR-STEP-beta}
of Algorithm \ref{CG-PR-STEP}, 

\begin{equation}
\beta={\displaystyle \frac{\left\langle \mathbf{g}_{k+\nicefrac{1}{2}},\mathbf{A}^{T}\mathbf{A}\mathbf{p}_{k-1}\right\rangle }{\left\langle \mathbf{A}\mathbf{p}_{k-1},\mathbf{A}\mathbf{p}_{k-1}\right\rangle }}.\label{eq: beta_k}
\end{equation}

Combining this with Step \ref{CG-PR-STEP-new_p} of Algorithm \ref{CG-PR-STEP},
we have that

\begin{eqnarray}
\left\langle \mathbf{A}\boldsymbol{p}_{k},\mathbf{A}\boldsymbol{p}_{k-1}\right\rangle  & = & \left\langle \boldsymbol{p}_{k},\mathbf{A}{}^{T}\mathbf{A}\mathbf{p}_{k-1}\right\rangle \nonumber \\
 & =- & \left\langle \mathbf{g}_{k+\nicefrac{1}{2}},\mathbf{A}{}^{T}\mathbf{A}\mathbf{p}_{k-1}\right\rangle +\nonumber \\
 &  & \frac{\left\langle \mathbf{g}_{k+\nicefrac{1}{2}},\mathbf{A}{}^{T}\mathbf{A}\mathbf{p}_{k-1}\right\rangle }{\left\langle \mathbf{A}\mathbf{p}_{k-1},\mathbf{A}\mathbf{p}_{k-1}\right\rangle }\left\langle \mathbf{p}_{k-1},\mathbf{A}^{T}\mathbf{A}\mathbf{p}_{k-1}\right\rangle \label{eq:A_conjugate}\\
 & = & 0.\nonumber 
\end{eqnarray}
The remaining two steps of Algorithm \ref{CG-PR-STEP} do not change
the validity of (\ref{eq:A_conjugate}).
\end{proof}
In Step \ref{S-CG-perturbation} of Algorithm \ref{S-CG} the procedure
$\mathbf{x}_{k+\nicefrac{1}{2}}=\text{perturbed}(\mathbf{x}_{k})$
is called to add a perturbation to the vector $\mathbf{x}_{k}$ and
thus produce a new vector $\mathbf{x}_{k+\nicefrac{1}{2}}$ that is
generally better than $\mathbf{x}_{k}$ according to a secondary criterion.
Using (\ref{eq:superiorized}), this $\mathbf{x}_{k+\nicefrac{1}{2}}$
is expressed as

\begin{equation}
\mathbf{x}_{k+\nicefrac{1}{2}}=\mathbf{x}_{k}+\mathbf{u}_{k},\label{eq: Superiorization}
\end{equation}
where $\mathbf{u}_{k}=\gamma_{k}\mathbf{v}_{k}$ is the overall perturbation
at the superiorization stage. In the superiorization methodology,
the sequence of the perturbations $\left(\mathbf{u}_{k}\right){}_{k=0}^{\infty}$
for all $k\geq0$ in (\ref{eq: Superiorization}) are bounded and
the norm of the sequence is summable, that is

\begin{equation}
\sum_{k=0}^{\infty}\left\Vert \mathbf{u}_{k}\right\Vert _{2}<\infty.\label{eq: summable}
\end{equation}
In what follows we use $c$ be the Euclidean norm $||\mathbf{A}||_{2}$
of the matrix $\mathbf{A}$; that is
\begin{equation}
c=\sup_{\left\Vert \mathbf{x}\right\Vert _{2}\neq0}\frac{\left\Vert \mathbf{Ax}\right\Vert _{2}}{\left\Vert \mathbf{x}\right\Vert _{2}}\label{eq: lambda}
\end{equation}
 We note that

\begin{equation}
c^{2}=||\mathbf{A}{}^{T}\mathbf{A}||_{2},\label{eq: gamma}
\end{equation}
and that $c$ is the largest singular value of $\mathbf{A}$. Another
notation that we use in the rest of this paper is $\eta_{1}=\frac{1}{4c^{2}}$.
\begin{prop}
\label{prop :CG step}Suppose that the $\mathbf{x}_{k+\nicefrac{1}{2}}$
produced by the execution of Step \ref{S-CG-perturbation} of Algorithm
\ref{S-CG} is expressed as (\ref{eq: Superiorization}) in which
the $\mathbf{u}_{k}$ satisfies

\begin{equation}
\left\Vert \mathbf{u}_{k}\right\Vert _{2}=\eta_{0}\left\Vert \mathbf{g}_{k}\right\Vert _{2},\label{eq: up bnd u_k}
\end{equation}
with $0<\eta_{0}\leq\eta_{1}$ and 

\begin{equation}
\mathbf{g}_{k}=\mathbf{A}^{T}\left(\mathbf{A}\mathbf{x}_{k}-\mathbf{y}\right).\label{eq:gradient at x_k-1}
\end{equation}

Then, upon the execution of Step \ref{S-CG-call-CG-PR-step} of Algorithm
\ref{S-CG}, it is the case that 

\begin{equation}
2f(\mathbf{x}_{k+\nicefrac{1}{2}})-2f(\mathbf{x}_{k+1})\geq\frac{1}{16c^{2}}\left\Vert \mathbf{g}_{k}\right\Vert _{2}^{2}.\label{eq: reduced Obj}
\end{equation}
\end{prop}
\begin{proof}
We look at the details of what happens during the execution of Step
\ref{S-CG-call-CG-PR-step} of Algorithm \ref{S-CG}, which is a call
to Algorithm \ref{CG-PR-STEP}. It follows from (\ref{eq:fuai_x})
that
\begin{eqnarray}
 & 2f(\mathbf{x}_{k+\nicefrac{1}{2}})-2f(\mathbf{x}_{k+1})\label{eq: reduced Obj part1}\\
 & =\left\langle \mathbf{A}\mathbf{x}_{k+\nicefrac{1}{2}}-\mathbf{y},\mathbf{A}\mathbf{x}_{k+\nicefrac{1}{2}}-\mathbf{y}\right\rangle -\left\langle \mathbf{A}\mathbf{x}_{k+1}-\mathbf{y},\mathbf{A}\mathbf{x}_{k+1}-\mathbf{y}\right\rangle .\nonumber 
\end{eqnarray}

From Step \ref{CG-PR-STEP-image_update} of Algorithm \ref{CG-PR-STEP},
we see that,

\begin{equation}
\mathbf{x}_{k+1}=\mathbf{x}_{k+\nicefrac{1}{2}}+\alpha\mathbf{p}_{k}.\label{eq: updated x_k+1}
\end{equation}

Combining this with (\ref{eq: reduced Obj part1}) we have that
\begin{equation}
\begin{array}{ll}
2f(\mathbf{x}_{k+\nicefrac{1}{2}})-2f(\mathbf{x}_{k+1})\\
\;=\left\langle \mathbf{A}\mathbf{x}_{k+\nicefrac{1}{2}}-\mathbf{y},\mathbf{A}\mathbf{x}_{k+\nicefrac{1}{2}}-\mathbf{y}\right\rangle \\
\quad-\left\langle \left(\mathbf{A}\mathbf{x}_{k+\nicefrac{1}{2}}-\mathbf{y}\right)+\alpha\mathbf{A}\mathbf{p}_{k},\left(\mathbf{A}\mathbf{x}_{k+\nicefrac{1}{2}}-\mathbf{y}\right)+\alpha\mathbf{A}\mathbf{p}_{k}\right\rangle \\
\;=-2\alpha\left\langle \mathbf{A}\mathbf{x}_{k+\nicefrac{1}{2}}-\mathbf{y},\mathbf{A}\mathbf{p}_{k}\right\rangle -\alpha^{2}\left\Vert \mathbf{A}\mathbf{p}_{k}\right\Vert _{2}^{2}\\
\;=-2\alpha\left\langle \mathbf{g}_{k+\nicefrac{1}{2}},\mathbf{p}_{k}\right\rangle -\alpha^{2}\left\Vert \mathbf{A}\mathbf{p}_{k}\right\Vert _{2}^{2}.
\end{array}\label{eq: reduced Obj part1-1}
\end{equation}
From Steps \ref{CG-PR-STEP-h} and \ref{CG-PR-STEP-alpha} of Algorithm
\ref{CG-PR-STEP}, it follows that

\begin{equation}
\alpha=-\frac{\left\langle \mathbf{g}_{k+\nicefrac{1}{2}},\mathbf{p}_{k}\right\rangle }{\left\Vert \mathbf{A}\boldsymbol{p}_{k}\right\Vert _{2}^{2}}.\label{eq: alpha}
\end{equation}

Combining this with (\ref{eq: reduced Obj part1-1}), we know that
\begin{eqnarray}
 &  & 2f(\mathbf{x}_{k+\nicefrac{1}{2}})-2f(\mathbf{x}_{k+1})\nonumber \\
 & = & 2\frac{\left\langle \mathbf{g}_{k+\nicefrac{1}{2}},\mathbf{p}_{k}\right\rangle ^{2}}{\left\Vert \mathbf{A}\boldsymbol{p}_{k}\right\Vert _{2}^{2}}-\frac{\left\langle \mathbf{g}_{k+\nicefrac{1}{2}},\mathbf{p}_{k}\right\rangle ^{2}}{\left\Vert \mathbf{A}\boldsymbol{p}_{k}\right\Vert _{2}^{2}}\label{eq: reduced obj part2}\\
 & = & \frac{\left\langle \mathbf{g}_{k+\nicefrac{1}{2}},\mathbf{p}_{k}\right\rangle ^{2}}{\left\Vert \mathbf{A}\mathbf{p}_{k}\right\Vert _{2}^{2}}.\nonumber 
\end{eqnarray}

Now we calculate $\left\langle \mathbf{g}_{k+\nicefrac{1}{2}},\mathbf{p}_{k}\right\rangle $.
By Step \ref{CG-PR-STEP-new_p} of Algorithm \ref{CG-PR-STEP}, we
have that

\begin{equation}
\mathbf{p}_{k}=-\mathbf{g}_{k+\nicefrac{1}{2}}+\beta\mathbf{p}_{k-1}.\label{eq: p_k iter}
\end{equation}
It follows that

\begin{equation}
\left\langle \mathbf{g}_{k+\nicefrac{1}{2}},\mathbf{p}_{k}\right\rangle =-\left\Vert \mathbf{g}_{k+\nicefrac{1}{2}}\right\Vert _{2}^{2}+\left\langle \mathbf{g}_{k+\nicefrac{1}{2}},\beta\mathbf{p}_{k-1}\right\rangle .\label{eq:< g_z,k, p_k> part1}
\end{equation}

By (\ref{eq:gradient at x_k-1}) and (\ref{eq: Superiorization})
we have that

\begin{equation}
\mathbf{g}_{k+1/2}=\mathbf{g}_{k}+\mathbf{A}^{T}\mathbf{\mathbf{A}}\mathbf{u}_{k}.\label{eq:g_z,k and g_k}
\end{equation}

From (\ref{eq: <g, p=003009=00003D0}) of Proposition \ref{prop: ortho gp}
with $n=k-1$, we have that

\begin{equation}
\left\langle \mathbf{g}_{k},\beta\mathbf{p}_{k-1}\right\rangle =0.\label{eq: g_k orthogonal to p_k-1}
\end{equation}

Combining these two facts, we have that

\begin{eqnarray}
 &  & \left\langle \mathbf{g}_{k+1/2},\beta\mathbf{p}_{k-1}\right\rangle \nonumber \\
 & = & \left\langle \mathbf{g}_{k},\beta\mathbf{p}_{k-1}\right\rangle +\left\langle \mathbf{A}^{T}\mathbf{\mathbf{A}}\mathbf{u}_{k},\beta\mathbf{p}_{k-1}\right\rangle \nonumber \\
 & = & \left\langle \mathbf{\mathbf{A}}\mathbf{u}_{k},\beta\mathbf{A}\mathbf{p}_{k-1}\right\rangle \label{eq:< g_z,k, p_k-1> part 1}\\
 & \leq & \left\Vert \mathbf{A}\mathbf{u}_{k}\right\Vert _{2}\cdot\left\Vert \beta\mathbf{A}\mathbf{p}_{k-1}\right\Vert _{2}.\nonumber 
\end{eqnarray}

It follows from Step \ref{CG-PR-STEP-new_p} of Algorithm \ref{CG-PR-STEP}
that 

\begin{equation}
\mathbf{A}\mathbf{p}_{k}=-\mathbf{A}\mathbf{g}_{k+\nicefrac{1}{2}}+\beta\mathbf{A}\mathbf{p}_{k-1}.\label{eq: Rp_k}
\end{equation}

Combining this with the orthogonality expressed in (\ref{eq: <p'_k, p'_k-1>=00003D0})
of Proposition \ref{prop: otrho conj}, we have that

\begin{equation}
\left\Vert \mathbf{A}\mathbf{g}_{k+\nicefrac{1}{2}}\right\Vert _{2}^{2}=\left\Vert \mathbf{A}\mathbf{p}_{k}\right\Vert _{2}^{2}+\left\Vert \beta\mathbf{A}\mathbf{p}_{k-1}\right\Vert _{2}^{2}.\label{eq: Pytha conj}
\end{equation}
Consequently, 
\begin{equation}
\left\Vert \mathbf{A}\mathbf{p}_{k}\right\Vert _{2}\leq\left\Vert \mathbf{A}\mathbf{g}_{k+\nicefrac{1}{2}}\right\Vert _{2},\label{eq: Up Bnd Rp_k}
\end{equation}
and

\begin{equation}
\left\Vert \beta\mathbf{A}\mathbf{p}_{k-1}\right\Vert _{2}\leq\left\Vert \mathbf{A}\mathbf{g}_{k+\nicefrac{1}{2}}\right\Vert _{2}.\label{eq: Up Bnd 1}
\end{equation}

From this and (\ref{eq:< g_z,k, p_k-1> part 1}), we know that

\begin{eqnarray}
\left\langle \mathbf{g}_{k+\nicefrac{1}{2}},\beta\mathbf{p}_{k-1}\right\rangle  & \leq & \left\Vert \mathbf{\mathbf{A}}\mathbf{u}_{k}\right\Vert _{2}\cdot\left\Vert \mathbf{A}\mathbf{g}_{k+\nicefrac{1}{2}}\right\Vert _{2}\nonumber \\
 & \leq & c^{2}\left\Vert \mathbf{u}_{k}\right\Vert _{2}\cdot\left\Vert \mathbf{g}_{k+\nicefrac{1}{2}}\right\Vert _{2}\label{eq: < g_z,k, p_k-1> part 2}\\
 & = & c^{2}\eta_{0}\left\Vert \mathbf{g}_{k}\right\Vert _{2}\cdot\left\Vert \mathbf{g}_{k+\nicefrac{1}{2}}\right\Vert _{2}.\nonumber 
\end{eqnarray}
The second inequality of (\ref{eq: < g_z,k, p_k-1> part 2}) comes
from (\ref{eq: lambda}) and the equality comes from (\ref{eq: up bnd u_k}).

By (\ref{eq:g_z,k and g_k}), (\ref{eq: gamma}) and (\ref{eq: up bnd u_k}),
we have that

\begin{eqnarray}
\left\Vert \mathbf{g}_{k+\nicefrac{1}{2}}\right\Vert _{2} & \geq & \left\Vert \mathbf{g}_{k}\right\Vert _{2}-\left\Vert \mathbf{A}^{T}\mathbf{\mathbf{A}}\mathbf{u}_{k}\right\Vert _{2}\nonumber \\
 & \geq & \left\Vert \mathbf{g}_{k}\right\Vert _{2}-c^{2}\eta_{0}\left\Vert \mathbf{g}_{k}\right\Vert _{2}\label{eq: bound norm g_zk}\\
 & = & \left(1-c^{2}\eta_{0}\right)\left\Vert \mathbf{g}_{k}\right\Vert _{2}.\nonumber 
\end{eqnarray}

From this and and the assumption that $0<\eta_{0}\leq\frac{1}{4c^{2}}$,
it follows that

\begin{equation}
\left\Vert \mathbf{g}_{k+\nicefrac{1}{2}}\right\Vert _{2}\geq\frac{3\left\Vert \mathbf{g}_{k}\right\Vert _{2}}{4}\geq\frac{\left\Vert \mathbf{g}_{k}\right\Vert _{2}}{2}.\label{eq: bound norm g_zk-part2}
\end{equation}

From this and (\ref{eq: < g_z,k, p_k-1> part 2}), it follows that

\begin{equation}
\left\langle \mathbf{g}_{k+\nicefrac{1}{2}},\beta\mathbf{p}_{k-1}\right\rangle \leq2c^{2}\eta_{0}\left\Vert \mathbf{g}_{k+\nicefrac{1}{2}}\right\Vert _{2}^{2}.\label{eq: < g_z,k, p_k-1> part 3}
\end{equation}

From this and (\ref{eq:< g_z,k, p_k> part1}), it follows that

\begin{equation}
\left\langle \mathbf{g}_{k+\nicefrac{1}{2}},\mathbf{p}_{k}\right\rangle \leq-(1-2c^{2}\eta_{0})\left\Vert \mathbf{g}_{k+\nicefrac{1}{2}}\right\Vert _{2}^{2}.\label{eq: bnd inp g_zk and p_k}
\end{equation}
Combining this with the assumption that $0<\eta_{0}\leq\frac{1}{4c^{2}}$,
we have that

\begin{equation}
\left\langle \mathbf{g}_{k+\nicefrac{1}{2}},\mathbf{p}_{k}\right\rangle ^{2}\geq\frac{1}{4}\left\Vert \mathbf{g}_{k+\nicefrac{1}{2}}\right\Vert _{2}^{4}.\label{eq: bnd sq inp g_zk and p_k}
\end{equation}
From this and (\ref{eq: reduced obj part2}) and (\ref{eq: Up Bnd Rp_k}),
it follows that

\begin{eqnarray}
2f(\mathbf{x}_{k+\nicefrac{1}{2}})-2f(\mathbf{x}_{k+1}) & = & \frac{\left\langle \mathbf{g}_{k+\nicefrac{1}{2}},\mathbf{p}_{k}\right\rangle ^{2}}{\left\Vert \mathbf{A}\mathbf{p}_{k}\right\Vert _{2}^{2}}\nonumber \\
 & \geq & \frac{\left\Vert \mathbf{g}_{k+\nicefrac{1}{2}}\right\Vert _{2}^{4}}{4\left\Vert \mathbf{A}\mathbf{g}_{k+\nicefrac{1}{2}}\right\Vert _{2}^{2}}\label{eq: reduced obj part 3}\\
 & \geq & \frac{\left\Vert \mathbf{g}_{k+\nicefrac{1}{2}}\right\Vert _{2}^{2}}{4c^{2}}.\nonumber 
\end{eqnarray}

From this and (\ref{eq: bound norm g_zk-part2}), (\ref{eq: reduced Obj})
follows.\end{proof}
\begin{lem}
\label{Lem: Overall step}Let $\eta_{2}$ be the positive solution
to the quadratic equation on the variable $\eta$:

\begin{equation}
\left(2+c^{2}\eta\right)\eta=\frac{1}{32c^{2}}.\label{eq: equation of coef}
\end{equation}

During the execution of Algorithm \ref{S-CG}, suppose that, for some
integer $k$, immediately after the execution of Step \ref{S-CG-perturbation}
of Algorithm \ref{S-CG}, the returned $\mathbf{x}_{k+\nicefrac{1}{2}}$
is expressed as (\ref{eq: Superiorization}) in which the $\mathbf{u}_{k}$
satisfies (\ref{eq: up bnd u_k}) with $0<\eta_{0}\leq\eta_{l}=\min\left\{ \eta_{1},\eta_{2}\right\} $
and $\mathbf{g}_{k}$ defined in (\ref{eq:gradient at x_k-1}). Then,
upon the execution of Step \ref{S-CG-call-CG-PR-step} of Algorithm
\ref{S-CG}, it is the case that 

\begin{equation}
2f(\mathbf{x}_{k})-2f(\mathbf{x}_{k+1})\geq\frac{1}{32c^{2}}\left\Vert \mathbf{g}_{k}\right\Vert _{2}^{2}.\label{eq: reduced Obj-1}
\end{equation}
\end{lem}
\begin{proof}
By (\ref{eq:fuai_x}) and (\ref{eq: Superiorization}), we have that
immediately after the execution of Step \ref{S-CG-perturbation} of
Algorithm \ref{S-CG}

\begin{eqnarray}
 &  & 2f(\mathbf{x}_{k})-2f(\mathbf{x}_{k+\nicefrac{1}{2}})\nonumber \\
 & = & \left\langle \mathbf{A}\mathbf{x}_{k}-\mathbf{y},\mathbf{A}\mathbf{x}_{k}-\mathbf{y}\right\rangle -\left\langle \mathbf{A}\mathbf{x}_{k+\nicefrac{1}{2}}-\mathbf{y},\mathbf{A}\mathbf{x}_{k+\nicefrac{1}{2}}-\mathbf{y}\right\rangle \nonumber \\
 & = & -2\left\langle \mathbf{A}\mathbf{x}_{k}-\mathbf{y},\mathbf{A}\mathbf{u}_{k}\right\rangle -\left\Vert \mathbf{A}\mathbf{u}_{k}\right\Vert _{2}^{2}.\label{eq: dec obj supr}
\end{eqnarray}

Combining this with (\ref{eq:gradient at x_k-1}), we know that

\begin{eqnarray}
2f(\mathbf{x}_{k})-2f(\mathbf{x}_{k+\nicefrac{1}{2}}) & = & -2\left\langle \mathbf{g}_{k},\mathbf{u}_{k}\right\rangle -\left\Vert \mathbf{A}\mathbf{u}_{k}\right\Vert _{2}^{2}.\label{eq: dec obj supr-1}
\end{eqnarray}

It follows that

\begin{eqnarray}
\left|2f(\mathbf{x}_{k})-2f(\mathbf{x}_{k+\nicefrac{1}{2}})\right| & \leq & 2\left|\left\langle \mathbf{g}_{k},\mathbf{u}_{k}\right\rangle \right|+\left\Vert \mathbf{A}\mathbf{u}_{k}\right\Vert _{2}^{2}\nonumber \\
 & \leq & 2\left\Vert \mathbf{g}_{k}\right\Vert _{2}\cdot\left\Vert \mathbf{u}_{k}\right\Vert _{2}+c^{2}\left\Vert \mathbf{u}_{k}\right\Vert _{2}^{2}\label{eq: diff obj perturbation}\\
 & = & \left(2+c^{2}\eta_{0}\right)\eta_{0}\left\Vert \mathbf{g}_{k}\right\Vert _{2}^{2}.\nonumber 
\end{eqnarray}

The second inequality of (\ref{eq: diff obj perturbation}) comes
from (\ref{eq: lambda}) and the equality comes from (\ref{eq: up bnd u_k}).

Let 

\begin{equation}
\varphi(\eta)=\left(2+c^{2}\eta\right)\eta.\label{eq: function 2}
\end{equation}
By the definition of $\eta_{2}$, we know that $\varphi(\eta_{2})=\frac{1}{32c^{2}}$.
Since $\varphi(\eta)$ is a monotone increasing function of $\eta$
for $\eta>0$, and $\varphi(0)=0$, $\eta_{0}\leq\eta_{2}$, we have
that

\begin{equation}
0<\varphi(\eta_{0})\leq\varphi(\eta_{2})=\frac{1}{32c^{2}}.\label{eq: coef f  bound}
\end{equation}
Combining this with (\ref{eq: diff obj perturbation}), it follows
that

\begin{equation}
\left|2f(\mathbf{x}_{k+1/2})-2f(\mathbf{x}_{k})\right|\leq\frac{1}{32c^{2}}\left\Vert \mathbf{g}_{k}\right\Vert _{2}^{2}.\label{eq: diff obj perturbation 2}
\end{equation}

From this and (\ref{eq: reduced Obj}) obtained by Proposition \ref{prop :CG step},
it follows that, upon the execution of Step \ref{S-CG-call-CG-PR-step}
of Algorithm \ref{S-CG} ,

\begin{eqnarray}
 &  & 2f(\mathbf{x}_{k})-2f(\mathbf{x}_{k+1})\nonumber \\
 & = & \left[2f(\mathbf{x}_{k})-2f(\mathbf{x}_{k+\nicefrac{1}{2}})\right]+\left[2f(\mathbf{x}_{k+\nicefrac{1}{2}})-2f(\mathbf{x}_{k+1})\right]\nonumber \\
 & \geq & -\left|2f(\mathbf{x}_{k})-2f(\mathbf{x}_{k+\nicefrac{1}{2}})\right|+\left[2f(\mathbf{x}_{k+\nicefrac{1}{2}})-2f(\mathbf{x}_{k+1})\right]\nonumber \\
 & \geq & -\frac{1}{32c^{2}}\left\Vert \mathbf{g}_{k}\right\Vert _{2}^{2}+\frac{1}{16c^{2}}\left\Vert \mathbf{g}_{k}\right\Vert _{2}^{2}\label{eq: overall reduced Obj}\\
 & \geq & \frac{1}{32c^{2}}\left\Vert \mathbf{g}_{k}\right\Vert _{2}^{2}.\nonumber 
\end{eqnarray}

\end{proof}
Now we are ready to complete the proof of \textbf{Theorem \ref{Thm: Conv-1}}.
\begin{proof}
Since $\mathbf{x}_{LS}$ is a minimizer of $f(\mathbf{x})$ if and
only if 

\begin{equation}
\mathbf{g}(\mathbf{x}_{LS})=\mathbf{A}^{T}\left(\mathbf{A}\mathbf{x}_{LS}-\mathbf{y}\right)=0,\label{eq: zero grad}
\end{equation}
we know that for any $\mathbf{x}$ such that $f(\mathbf{x})>\varepsilon_{0}$,
$\left\Vert \mathbf{g}(\mathbf{x})\right\Vert _{2}>0$. 

Given a positive $\varepsilon$ such that $\varepsilon>\varepsilon_{0}$,
define 

\begin{equation}
\theta=\underset{\boldsymbol{x}}{\inf}\left\{ \left\Vert \mathbf{g}(\mathbf{x})\right\Vert _{2}\,|\,f(\mathbf{x})>\varepsilon\right\} .\label{eq: infimum grad}
\end{equation}
Then $\theta>0$.

Let $c$ be as specified in (\ref{eq: lambda}), $\eta_{1}=\frac{1}{4c^{2}}$,
and $\eta_{2}$ be the positive solution to the quadratic equation
(\ref{eq: equation of coef}) of $\eta$, $\eta_{l}=\min\left\{ \eta_{1},\eta_{2}\right\} $.
During the execution of\textbf{ }Algorithm \ref{S-CG}, the sequence
of the perturbations $\left(\mathbf{u}_{k}\right){}_{k=1}^{\infty}$
for all $k\geq1$ in (\ref{eq: Superiorization}) generated by execution
of Step \ref{S-CG-perturbation} of Algorithm \ref{S-CG},\textbf{
}satisfies (\ref{eq: summable}). Hence, there exists an integer $K$
such that $\left\Vert \mathbf{u}_{k}\right\Vert _{2}\leq\eta_{l}\theta$
for all $k\geq K$. Let $\mathbf{S}=\left\{ 2f(\mathbf{x}_{k})\,|\:k\geq K\right\} $
be the sequence of $2f(\mathbf{x}_{k})$ of which the $\mathbf{x}_{k}$
for $k\geq K$ is generated by an execution of Step \ref{S-CG-call-CG-PR-step}
of Algorithm \ref{S-CG}. By Lemma \ref{Lem: Overall step}, we know
that $0<2f(\mathbf{x}_{k+1})<2f(\mathbf{x}_{k})$, hence the sequence
in $\mathbf{S}$ is strictly decreasing and bounded.

We prove this theorem with the method of proof by contradiction. Suppose
that the Algorithm \ref{S-CG} never terminates. Then for any $k\geq K$,
it holds that

\begin{equation}
f(\mathbf{x}_{k})>\varepsilon.\label{eq: not terminated}
\end{equation}

Let 

\begin{equation}
\xi=\inf\left\{ 2f(\mathbf{x}_{k})\,\mid\,k\geq K\right\} .\label{eq: low bound dist}
\end{equation}
Now we prove that if the Algorithm \ref{S-CG} never terminates, then
there exists a $k\geq K$ such that 

\begin{equation}
2f(\mathbf{x}_{k})<\xi,\label{eq: contradiction}
\end{equation}
which contradicts the definition of $\xi$.

Let $\delta=\frac{\theta^{2}}{32c^{2}}$. By the definition of $\xi$,
we know that there exists an integer $k\geq K$ such that 

\begin{equation}
2f(\mathbf{x}_{k})\leq\xi+\frac{\delta}{2}.\label{eq: near low bound}
\end{equation}

From the facts that $\left\Vert \mathbf{u}_{k}\right\Vert _{2}\leq\eta_{l}\theta$
and $\left\Vert \mathbf{g}_{k}\right\Vert _{2}\geq\theta$ for $k\geq K$
with $\mathbf{g}_{k}$ defined in (\ref{eq:gradient at x_k-1}), we
know that the precondition of Lemma \ref{Lem: Overall step} is satisfied
for $k\geq K$. It follows from Lemma \ref{Lem: Overall step} that 

\begin{eqnarray}
2f(\mathbf{x}_{k+1}) & \leq & 2f(\mathbf{x}_{k})-\frac{1}{32c^{2}}\left\Vert \mathbf{g}_{k}\right\Vert _{2}^{2}\nonumber \\
 & \leq & 2f(\mathbf{x}_{k})-\delta\label{eq: lower}\\
 & \leq & \xi-\frac{\delta}{2}.\nonumber 
\end{eqnarray}
The second inequality of (\ref{eq: lower}) comes from the facts that
$\left\Vert \mathbf{g}_{k}\right\Vert _{2}\geq\theta$ and $\delta=\frac{\theta^{2}}{32c^{2}}$.
The third inequality of (\ref{eq: lower}) comes from (\ref{eq: near low bound}).
Then we know that (\ref{eq: contradiction}) holds for $k+1$, which
contradicts the definition of $\xi$. Hence\textbf{ }Algorithm \ref{S-CG}
terminates.
\end{proof}

\section*{References{\normalsize{}\label{sec:References}}}

\bibliographystyle{iopart-num}
\bibliography{References_Gabor_8_17_17}

\end{document}